\documentclass[ssy]{imsart}
\setattribute{journal}{name}{}
\usepackage{a4wide}
\usepackage{amsmath,amsfonts}
\usepackage{amsthm}
\usepackage{enumerate}
\usepackage{dsfont}
\usepackage[colorlinks=true,urlcolor=blue,citecolor=blue]{hyperref}
\usepackage[usenames,dvipsnames]{color}

\usepackage{placeins}
\usepackage{graphicx}
\graphicspath{{./images/}}

\usepackage{longtable}

\usepackage{natbib}

\theoremstyle{plain} 
\newtheorem{theorem}{Theorem}[section]
\newtheorem{lemma}[theorem]{Lemma}

\newtheorem{corollary}[theorem]{Corollary}

\theoremstyle{definition} 
\newtheorem{definition}[theorem]{Definition}

\newtheorem{example}[theorem]{Example}

\newtheorem{remark}{Remark}

\newcommand{\R}{\mathbb{R}}
\newcommand{\N}{\mathbb{N}}

\newcommand{\ind}[1]{\mathds{1}_{\left[ #1 \right] }}
\newcommand{\eqd}{\,{\buildrel d \over =}\,}
\newcommand{\diff}[1]{\mathrm{d} #1}
\newcommand{\RV}[1]{\mathcal{RV}_{ #1 }}
\newcommand{\D}{\mathbb{D}}

\newcommand{\Ceq}[1]{\mathbb{C}_{= #1 }}

\newcommand{\M}{\mathbb{M}}
\newcommand{\MO}{\M_\mathbb{O}}
\newcommand{\Rjzero}{\left(\R\backslash\{0\}\right)^j}
\newcommand{\thenum}{^{\text{th}}}
\newcommand{\Leb}{\mathrm{Leb}}
\newcommand{\LD}[1]{ \mathrm{LD} \left( #1 \right)}

\newcommand{\Prob}{\mathrm{P}}
\renewcommand{\P}[1]{\Prob \left( #1 \right)}
\newcommand{\E}{\mathbb{E}}
\newcommand{\bE}[1]{\mathbb{E}\left[ #1 \right]}
\newcommand{\var}[1]{\mathrm{Var}\left[ #1 \right]}

\newcommand{\bfS}{\mathbb{S}}

\newcommand{\bK}{\boldsymbol{K}}
\newcommand{\bk}{\boldsymbol{k}}
\newcommand{\bI}{\boldsymbol{I}}

\newcommand{\vague}{\stackrel{\lower0.2ex\hbox{$\scriptscriptstyle
                    \it{v} $}}{\rightarrow}}
\newcommand{\weak}{\stackrel{\lower0.2ex\hbox{$\scriptscriptstyle
                    \it{w} $}}{\rightarrow}}
\newcommand{\what}{\stackrel{\lower0.2ex\hbox{$\scriptscriptstyle
                    \it{\hat{w}} $}}{\rightarrow}}
\newcommand{\eqdis}{\stackrel{\lower0.2ex\hbox{$\scriptscriptstyle
                    \mathrm{d}$}}{=}}
\newcommand{\distr}{\stackrel{\lower0.2ex\hbox{$\scriptscriptstyle
                    \it{d} $}}{\rightarrow}}

\def\bA{\boldsymbol A}
\def\bB{\boldsymbol B}

\def\bQ{\boldsymbol Q}
\def\bX{\boldsymbol X}

\def\bZ{\boldsymbol Z}

\def\bu{\boldsymbol u}

\def\bS{\boldsymbol S}

\def\bz{\boldsymbol z}

\allowdisplaybreaks 

\definecolor{darkgreen}{RGB}{0,139,0}

\arxiv{math.PR/0000000}

\begin{document}

\begin{frontmatter}
\title{Heavy-tailed random walks, buffered queues and hidden large deviations}
\runtitle{Heavy-tailed buffered queues}

\begin{aug}
\author{\fnms{Harald} \snm{Bernhard} \ead[label=e1]{harald\_bernhard@mymail.sutd.edu.sg}}
 \and
  \author{\fnms{Bikramjit}  \snm{Das} \ead[label=e2]{bikram@sutd.edu.sg}}

  \runauthor{H. Bernhard \and B. Das}

  \affiliation{Singapore University of Technology and Design}

  \address{Singapore University of Technology and Design,\\ Pillar of Engineering Systems and Design,\\8 Somapah Road, Singapore 487372 \\
           \printead{e1}\\ \printead{e2}}

\end{aug}

\begin{abstract}
  It is well-known that large deviations of random walks driven by independent and identically distributed heavy-tailed random variables are governed by the so-called principle of one large jump.
  We note that further subtleties hold for such random walks in the large deviation scale which we call hidden large deviation. We apply this idea in the context of queueing processes with heavy-tailed
  service times and study approximations of severe congestion times for (buffered) queues. We conclude with simulated examples to verify our results.

\end{abstract}



\end{frontmatter}

%
%
%
%
%
%
%
%
%
%
\section{Introduction}
Stochastic processes with heavy-tailed components as building blocks are of interest in many areas of application, including, but not restricted to,  hydrology \citep{anderson:meerschaert:1998}, finance, insurance and risk management
  \citep{smith:2003,embrechts1997modellingBook,ibragimov:jaffee:walden:2011},  tele-traffic data \citep{crovella:bestavros:taqqu:1998}, queueing theory \citep{boxma:cohen:1999,zwart2001:bptail}, social networks and random
  graphs \citep{ bollobas:borgs:chayes:riordan:2003, durrett:2010}. The notion of heavy-tails in applied probability is often studied under the paradigm of regular variation. In this paper we concentrate on understanding subtle properties of heavy-tailed random walks which enables us to understand structures of simple GI/G/1 queues with heavy-tailed service times under certain regularity conditions. In particular, we establish how queueing congestion may occur (which we define in terms of long intense periods), not only because of one large jump, but also in terms of further jumps occurring in the process.

  It is well-known that if $\{Z_i\}_{i\geq 1}$  are iid zero mean  regularly varying random variables, then large deviations of their partial sums  $S_n = \sum_{i=1}^n Z_i$ are essentially due to one of the random variables $Z_i$ attaining a large value, 
see \cite[Section 8.6]{embrechts1997modellingBook} for further details. Early results on this notion popularly known as the \emph{principle of one large jump} were obtained in \cite{nagaev:1969a,nagaev:1969b,nagaev:1969c}. More formally, the notion of one large jump in this case  can be written as
\begin{align}
	\label{eqn_nagaev}
	\P{|S_n| > x} = n\P{|Z_1|>x} (1+ o(1)), \quad   x>b_n
\end{align}
for some choice of $b_n \uparrow \infty$ as $n\to \infty$. Similar large deviation principles (LDPs) have been obtained in \cite{denisov:dieker:shneer:2008} under the more general assumption of the random variables being sub-exponential. Moving forward from random variables, the notion of regular variation of c\`adl\`ag processes has been characterized in \cite{hult:lindskog:2005SPA}. It was aptly noted in \cite{hult2005} that large deviations for such processes with heavy-tailed margins are very closely related to the notion of regular variation.

A precise large deviation result for partial sum processes on the space $\D[0,1]$ of c\`adl\`ag functions was provided in \citep[Theorem 2.1]{hult2005}; in fact this result was obtained for $d$-dimensional processes. In particular for $d=1$, the authors establish  that if $S^{n} =(S_{\lfloor nt\rfloor})_{t\in{[0,1]}}$
 is the c\`adl\`ag embedding of $\{S_{k}\}_{k=1}^n$ into $\D[0,1]$ with $S_{0}=0$, for suitably chosen sequences $\gamma_n>0$ and $\lambda_n \uparrow \infty$ one can observe that
\begin{align}
	\label{eqn_ld_principle_hult_et_al}
	\gamma_n \P{S^n /\lambda_n \in \;\cdot\; } \stackrel{ w^{\#}}{\to} \mu(\cdot ), \quad n\to \infty,
\end{align}
for a non-null measure $\mu$, where $\stackrel{w^{\#}}{\to}$ denotes convergence in the space of boundedly finite measures on $\overline{\D}_{0}$, see \cite[Appendix 2.6]{daley2003introPPvol1} and \cite{hult2005} for further details on the space and $w^{\#}$-convergence. In particular the result  shows that an appropriate  choice of scaling is $\gamma_{n} = [n \P{Z_1>\lambda_{n}}]^{{-1}}$ and  the limit measure $\mu$ concentrates all its mass on step functions with exactly one jump discontinuity, which essentially retrieves the \emph{one large jump principle}. 

Now it seems likely that there are possibilities - albeit rarer than the above case - that a large deviation of $S_{n}$ may occur because two or more of the random variables $\{Z_{i}\}_{i=1}^n$ were large. The probabilities of these events, although negligible under the scaling $\gamma_{n} = [n \P{Z_1>\lambda_{n}}]^{{-1}}$, are not exactly zero. Consequently, in this paper we aim to recover the rates at which such deviations happen and examine their structure. Furthermore our goal is to use such results in the context of queueing to understand the behavior of what we call long intense periods in a large-deviation-type event. 
Analysis of hidden behavior of regularly varying sequences on $\R^d$ \citep{resnick:2002a,das2013multidimEdge} and more recently on $\R^{{\infty}}$ and L\'evy processes on $\D[0,1]$   \citep{LindskogResnickRoy2014:probSurveys} has been conducted under the name \emph{hidden regular variation}. Connections between hidden regular variation and elements of the classical large deviations framework have been established recently in \cite{rhee2016:samplePathLDP}.

In this paper, our first contribution is to extend the large deviation principle in \eqref{eqn_ld_principle_hult_et_al} to \emph{hidden large deviations} in the spirit of hidden regular variation. We establish that the most probable way a large deviation event occurs, which is not the result of only one random variable being large, is actually when two random variables are large; resulting in a non-null limit measure as in \eqref{eqn_ld_principle_hult_et_al} concentrating on processes having two jump discontinuities. For our analysis we use the framework proposed in \cite{LindskogResnickRoy2014:probSurveys} and  the notion of convergence used here is known as  $\MO$-convergence which is closely related to the $w^{\#}$-convergence of boundedly finite measures and developed in \cite{hult:lindskog:2006a,das2013multidimEdge,LindskogResnickRoy2014:probSurveys}.  We briefly recall the required background on regular variation and $\MO$-convergence  in Section \ref{sec:background}. The results on hidden large deviations of random walks are dealt with in Section \ref{sec:hld}, where the key result is obtained in Theorem \ref{thm_large_dev}.

Queues with heavy-tailed service times have been of interest to researchers for a few years \citep{boxma:cohen:1999,jelenkovic1999:lossSubexp,zwart2001:bptail}. Our interest lies in figuring out when do we see long busy or intense periods in a queue.  In \cite{zwart2001:bptail}, the author shows that for a GI/G/1 queue with heavy-tailed service times, the most likely way a long busy period occurs is when one big service requirement arrives at the beginning of the busy period and the queue drifts back to zero linearly thereafter. Consequently, a large deviation of a queueing process also  looks exactly the same. 
\cite{jelenkovic1999:lossSubexp} studies the steady state loss in buffered queues and shows that for large buffers $K$ the steady state loss can be approximated by the expected loss due to one arrival $A$ filling the buffer completely starting zero, that is the expected loss is approximately $\E[A-K]$; see \cite{zwart2000} for similar results concerning fluid queues.

Equipped with a \emph{hidden large deviation principle} for random walks, in Section \ref{sec:queues} we study queuing processes with heavy-tailed service times and finite capacity, which is a natural model to assume in many contexts. We define a \emph{long intense period} as the fraction of time a queue with buffer capacity {$K>0$} spends continuously above a level $\theta K, \; \theta \in (0,1)$ for one sojourn and study the length of the longest such period for a given observation horizon.
A closely related notion of \emph{long strange segments}  \citep{mansfield2001} has been consequently investigated in \cite[Section 4]{hult2005}, which examines the length of time the average process value spends in an unexpected regime. Considering hidden large deviations in such a setting  provides more insight since we observe that the first large deviation approximation gives only a crude estimate of the distribution of the length of intense periods for large buffers.  
In Theorem \ref{thm:longintense} we derive an approximation to the distribution of the length of long intense periods in queues with large buffer sizes and conduct a simulation study in Section \ref{sec:simul} to show the effectiveness of the result. Finally, future directions are indicated and conclusions drawn in Section \ref{sec:conclusion}.

\section{Notations and background}
\label{sec:background}
In this section we provide a summary of frequently used notations and concepts along with a review of material necessary for the results in the following sections.  We mostly  adhere to the notations and definitions introduced in \cite{LindskogResnickRoy2014:probSurveys}.

\subsection{Basic notations}\label{subsec:notation}
A few notations and concepts are summarized here. 
Detailed discussions are in the references provided. Unless otherwise specified, capital letters like $X,Z, S$ with various subscripts and superscripts are reserved for real-valued (and sometimes vector-valued) random variables, whereas bold-symboled capital letters like $\bX,\bZ, \bS$ (again with various subscripts and superscripts) denote vector- or function-valued random elements. Small letters in bold, like $\bz$, are vectors in a suitable Euclidean space where $\bz=(z_{1},\ldots,z_{n})$ if $\bz\in \R^{n}$.  
%
%
%
%
%
\begin{longtable}{p{0.17\textwidth} p{0.77\textwidth}}
	$\RV{\beta}$ & 
		Regularly varying functions with index $\beta \in \R$: that is, $f:\mathbb{R}_+\mapsto \mathbb{R}_+$ satisfying $\lim\limits_{t\to\infty}f(tx)/f(t)=x^\beta,$ for $x>0.$  
		We abuse notation and write $X\in \RV{-\alpha}$ for regularly varying random variables as in Definition \ref{def:regvar}. In case of positive random variables, this is equivalent to requiring the tail of the cumulative distribution function to satisfy the limit relation above.
	\\[2mm]
	$\M(\mathbb S \backslash \mathbb C)$ & 
		$\M(\mathbb S, \mathbb C) = \M(\mathbb S \backslash \mathbb C)$ is the set of Borel measures on $\mathbb S \backslash \mathbb C$ that  are finite on sets bounded away from $\mathbb C$.
	\\[2mm]
	$\mu_n\to \mu$ & 
		Convergence in $\M(\mathbb{S}\backslash \mathbb{C})$; see  Definition \ref{def_conv_MO}.
	\\[2mm]
	$U_j^\uparrow$ & 
		$\{\bu \in {[0,1]^j}: \; 0\leq u_1 < \dots < u_j \leq 1 \}$.
	\\[2mm]
	$\D=\D([0,1],\R)$ & 
		Space of all real-valued c\`adl\`ag functions on $[0,1]$ equipped with the Skorohod $J_1$ metric.
	\\[2mm]
	$d_{J_1}$ &
		Metric on $\D$. If $\Lambda$ denotes the class of strictly increasing continuous functions $\lambda: [0,1] \to [0,1]$ with $\lambda(0)=0, \lambda(1)=1$, then for $f,g\in \D$, we define		
{\begin{align}\label{def:J1}
	\begin{split}
		d_{{J_{1}}} (f,g) :=& \inf_{\lambda\in \Lambda} \max \left\{ \sup_{t\in[0,1]} |f(t)-g\circ\lambda (t)|, \sup_{{t\in [0,1]} } |\lambda(t)-t| \right\}\\ 
		=& \inf_{\lambda\in \Lambda} \|f-g\circ\lambda\| \vee \|\lambda-e\|
	\end{split}
\end{align}}where $e(t)=t, \forall t\in [0,1]$. See \cite{billingsley:1999} for further details on the $J_{1}$-topology. Note that similar definitions could be worked out if we take $\D_{M} =\D([0,M],\R)$ in place of $\D$. 
	\\[2mm]
	$\D_{{=j}}$ & 
		Space of all real-valued step functions on [0,1] with exactly $j$ jumps, $j\geq 1$. Assume $\D_{=0}$ is the space containing only the constant function at $0$. Moreover, $\D_{=j}\subset \D$.
	\\[2mm]
	$\D_{{\le j}}$ & 
		Space of all step functions on [0,1] with  $j$ or less jumps, $\D_{{\le j}}=\bigcup\limits_{{k=0}}^{j} \D_{=k}$.
	\\[2mm]
	$\Ceq k (\lambda)$ & 
		$\{ \bz\in \R^n: \; |\{i: |z_i| > \lambda \}| = k\} \quad \text{for} \quad \lambda>0, 1\le k\le n$. It denotes the subset of $\R^{n}$ where exactly $k$ co-ordinates are above $\lambda > 0$ in absolute value. Moreover $\R^{n} = \bigcup\limits_{k=0}^{n } \Ceq k (\lambda)$. See Section \ref{sec:hld}.
	\\[2mm]
	$\nu_{\alpha}^{j}$ & 
		Product measure on $(\R\backslash \{0\})^{j}:\;\underbrace{\nu_{\alpha}\times \ldots \times   \nu_{\alpha}}_{j\; \text{times}} $ with $\nu_{\alpha}$ as defined in \eqref{def:nualpha}.
\end{longtable}


\subsection{Convergence in $\MO$}

We state our results as convergence in $\MO$, a mode of convergence closely related to standard weak convergence of probability measures. 
The idea is as follows: to allow for hidden large deviations we need to exclude the set of possible large deviations from the space we consider in order to keep the limit measure non-degenerate in that region. This is similar to how large deviations avoid the law of large numbers for zero mean random variates: we need to exclude $0$ from the non-negative real line to obtain a limit measure for $\P{|Z_1 + \cdots + Z_n| \geq nz},\; z> 0$. Convergence in $\MO$ follows the same principle but we allow for the removal of an arbitrary closed set. As a consequence, we can define convergence in c\`adl\`ag spaces where we exclude certain types of step functions which form a closed set in $\D$.

In particular, let $\mathbb S$ be a complete separable metric space, $\mathcal B(\mathbb S)$ the collection of Borel sets on $\mathbb S$ and $\mathbb C \subset \mathbb S$ a closed subset of $\mathbb S$. Then we denote $\M(\mathbb S, \mathbb C) = \M(\mathbb S \backslash \mathbb C)$ the set of Borel measures on $\mathbb S \backslash \mathbb C$ which are finite on sets bounded away from $\mathbb C$,
that is, the collection of  sets $A \in \mathcal B(\mathbb S)$ such that $\inf \{ d(x,y): x \in \mathbb C, \; y \in A\}>0$ where $d$ denotes the metric on $\mathbb S$. 
Finally, we call a sequence $\{ \mu_n \}_{n\geq 0} $ convergent if the assigned values converge for a suitable class of test functions or sets. Denoting $\mathbb O = \mathbb S \backslash \mathbb C$ the support set we use the notation $\MO := \M(\mathbb S, \mathbb C)$ as eponym for the mode of convergence and use the following definition of  $\MO$-convergence.
\begin{definition}\label{def_conv_MO}
A sequence of measures $\{\mu_n\}_{n\geq 0} \subset \MO$ converges to  $\mu_0 \in \MO$ if for all closed sets $F$ and open sets $G$ in $\mathcal{B}(\bfS)$ which are bounded away from $\mathbb C$ we have 
\begin{align*}
	\limsup_{n\to \infty} \mu_n (F)  \leq \mu_0(F), \\
	\liminf_{n \to \infty} \mu_n(G) \geq \mu_0(G).
\end{align*}
We write $\mu_n \to \mu_0$ in $\MO$ as $\; n \to \infty$, or simply $\mu_n \to \mu_0$.
\end{definition}

The definition above only states $\MO$-convergence in terms of open and closed subsets of $\mathbb O$. Theorem 2.1 of \cite{LindskogResnickRoy2014:probSurveys} provides several alternative characterizations of convergence in $\MO$. The corresponding version of a continuous mapping theorem in $\MO$ follows as Theorem 2.3 in the same publication. We state the continuous mapping theorem again for the sake of convenience. We denote a second space with the same properties as $\mathbb S$ by $\mathbb S'$ and similarly add dashes to the corresponding elements of the second space.

\begin{theorem}[\cite{LindskogResnickRoy2014:probSurveys}, Theorem 2.3]
	\label{thm:cont_map}
	Denote $h: \mathbb S \to \mathbb S'$ a measurable mapping such that for all sets $A' \in \mathcal B(\mathbb S') \cap h(\mathbb S \backslash \mathbb C)$ bounded away from $\mathbb C'$ also $h^{-1}(A')$ is bounded away from $\mathbb C$. Then $\mu_n \to \mu$ in $\M(\mathbb S, \mathbb C)$ implies $\mu_n\circ h^{-1} \to \mu \circ h^{-1}$ in $\M(\mathbb S',\mathbb C')$ if the measure $\mu$ attains zero mass on the set of discontinuity points of $h$.
\end{theorem}

\paragraph{Relationship between $\MO$-convergence and $w^\#$ convergence}

\cite{hult2005} state their version of a large deviation result for random walks in terms of $w^\#$ convergence, working with boundedly finite measures instead. Specifically they consider the space $\overline \D_0 = (0,\infty] \times \mathbb S_\D$, where $\mathbb S_\D$ is the unit sphere in $\D$. 
The metric on the radius is defined as $d_{(0,\infty]}(x,y):= |1/x - 1/y|$, thus making any set not bounded away from the zero function (in the usual $J_1$ metric) unbounded in the modified space. In turn, one may work with convergence of boundedly finite measures. Unfortunately it is not immediately clear how to extend this theory to allow for the removal of more than just the zero function, whereas convergence in $\MO$ on the contrary is specifically designed for this purpose.

\subsection{Regular variation and heavy-tailed large deviations}
A measurable function \linebreak $f: (0,\infty) \to (0,\infty)$ is regularly varying at infinity with index $\alpha \in \R$ if $\lim_{t\to \infty} {f(tx)}/{f(t)} = x^{\alpha}$ for all $x>0$. A sequence of positive numbers $\{a_{n}\}_{n\ge 1} $ is regularly varying with index $\alpha \in \R$ if $\lim_{{n\to\infty}} {a_{[cn]}}/{a_{n}} = c^{{\alpha}}$ for all $c>0$.
Regular variation of unbounded random variables thus usually is defined in terms of regular variation of the tail of the corresponding cumulative distribution functions at infinity; see \cite{bingham:goldie:teugels:1989}, \cite{FerreiraDeHaan2006:extreme}, or \cite{Resnick2007:htp} for related properties and examples.   We work with an equivalent definition stated in terms of $\MO$-convergence \cite[Section 3.2]{LindskogResnickRoy2014:probSurveys}.
\begin{definition}\label{def:regvar}
A random variable $X$ is  \emph{regularly varying} at infinity if there exists a regularly varying sequence $\gamma_n$ and a non-zero measure $\mu \in \M(\R\backslash\{0\})$ such that
\begin{align*}
	\gamma_n \P{X/n \in \cdot} \to \mu(\cdot), \quad n\to \infty,
\end{align*}
in $\M(\R\backslash\{0\})$. 
\end{definition}
Since the measure $\mu$ satisfies the scaling property $\mu(sA) = s^{-\alpha} \mu(A),\; s>0, A \in \R\backslash\{0\}$ for some {$\alpha \ge 0$} we write $X\in \RV{-\alpha}$.  For this paper we assume $\alpha>0$.  Moreover, for $X\in \RV{-\alpha}$, we also assume that the following condition is satisfied.
\[ \lim_{n\to \infty} \frac{\P{X>n}}{\P{|X|>n}} = p, \quad  \lim_{n\to \infty} \frac{\P{X<-n}}{\P{|X|>n}} = 1-p:=q.\]
for some $0\le p \le 1$. This is called the \emph{tail balance condition}. For univariate $X\in\RV{-\alpha}$ we stick to this choice unless otherwise specified. We also denote by $\nu_{\alpha}$, the following measure on $\R\backslash \{0\}$ for $x>0,y>0$,
\begin{align}\label{def:nualpha}
 \nu_{\alpha}((-\infty,-y) \cup (x,\infty)) = qy^{{-\alpha}} + px^{-\alpha}.
 \end{align}
The limit measure $\mu$ in Definition \ref{def:regvar} is equal to $\nu_{\alpha}$ if we choose the sequence $\gamma_n$ to be $ [\P{|X| > n}]^{{-1}}$.
Now, a heavy-tailed large deviation principle for real-valued random variables can be defined as follows in terms of $\MO$-convergence.
\begin{definition}\label{def:ldp}
A sequence of random variables $\{X_n\}_{{n\ge 1}}$, with $X_n \to 0$ in probability, satisfies a \emph{heavy-tailed large deviation principle} if there exists a positive sequence $\gamma_n \uparrow \infty $ and a non-zero measure $\mu \in \M(\R\backslash\{0\})$ such that as $n\to \infty$
\begin{align*}
	\gamma_n \P{X_n \in \cdot} \to \mu(\cdot),
\end{align*}
in $\M(\R\backslash\{0\})$.  
\end{definition}

\begin{remark}
The similarity between the definitions of regular variation and heavy-tailed large deviation principle (LDP) is quite evident here. One salient difference is that regular variation is defined for a single random element, whereas an LDP, for a sequence of random elements. The special case of $X_n = X/n$ shows that regular variation is a specific form of heavy tailed LDP according to our definition.
The definition of an LDP implies that $\P{X_n \in \cdot} \to 0$ as $n\to \infty$ for all sets in $\R\backslash\{0\}$.
\end{remark}

For our purposes we give a more general definition of LDPs for random elements on a metric space $\mathbb S$. Additionally we do no longer restrict to removing the zero element, but an arbitrary closed set $\mathbb C \subset \mathbb S$.

\begin{definition}\label{def:proc_ldp}
The random elements $\{\bX_n\}_{{n\ge 1}} \subset \mathbb S$ satisfy a \emph{(heavy-tailed) large deviation principle} on $\mathbb S \backslash \mathbb C$ for a closed set $\mathbb C \subset \mathbb S$ if there exists a positive sequence $\gamma_n \uparrow \infty $, and a non-zero measure $\mu \in \M(\mathbb S \backslash \mathbb C)$ such that as $n\to \infty$
\begin{align*}
	\gamma_n \P{\bX_{n} \in \; \cdot\;} \to \mu(\cdot),
\end{align*}
in $\M(\mathbb S \backslash \mathbb C)$.  We write $\bX_{n} \in \mathrm{LD}( \gamma_{n},\mu,\mathbb S \backslash \mathbb C)$.
\end{definition}

The definition of heavy-tailed LDP as given in \cite[Definition 1.3]{hult2005} is equivalent to Definition \ref{def:proc_ldp} for stochastic processes with sample paths in $\D$. It has been observed, especially in the case of heavy-tailed random walks, that the limit measure $\mu$ obtained in the LDP, concentrates only on step functions with one jump; see \citep[Theorem 2.1]{hult2005}. Hence we may ask whether a different structure is observable if  we do not allow one jump functions to be in the support of the limit  measure for an LDP. Essentially, we are asking how often do we see events which are not governed by one jump in the space $\D$. 
The same question can be asked iteratively by removing the support set of a new found limit measure and examining the hidden structure of rarer and rarer events.
Hence a \emph{sequence of large deviation principles} can be defined here.

\begin{definition}\label{def:proc_hldp}
The random elements $\{\bX_n\}_{{n\ge 1}} \subset \mathbb S$ satisfy a \emph{sequence of (heavy-tailed) large deviation principles} if there exists an increasing sequence $\{\mathbb C^{(j)}\}_{j\geq 1}$ of closed subsets of $\mathbb S$ (i.e. $\mathbb C^{(k)} \supset \mathbb C^{(j)}$ for $k>j\geq 1$), positive sequences $\gamma_n^{(j)} \uparrow_n \infty$, $j\geq 1$ with $\gamma_n^{(j+1)} / \gamma_n^{(j)} \to_n \infty$,  and non-zero measures $\mu^{(j)} \in \M(\mathbb S \backslash \mathbb C^{(j)}), j \geq 1$ such that 
\begin{align*}
	\bX_n \in \mathrm{LD}(\gamma_n^{(j)},\mu^{(j)},\mathbb S \backslash \mathbb C^{(j)}), \quad j\geq 1.
\end{align*}
\end{definition}
A similar definition could be stated for random elements satisfying only a finite number $J\geq j \geq 1$ of large deviation principles.


\begin{remark}
	The limit measure $\mu^{(j)}$ necessarily satisfies $\mu^{(j)}\left( (\mathbb C^{(j+1)})^c \right)=0$ for all $j\geq 1$. More precisely, the measure concentrates on $\mathbb C^{(j+1)} \backslash \mathbb C^{(j)}$. Thus, the $k\thenum$ level LDP uncovers the structure of rare events which were \emph{hidden} (i.e. negligible) under the scaling of the preceding $j\thenum$ level LDPs of the sequence with $j<k$. 
\end{remark}

\section{Hidden large deviations and random walks}
\label{sec:hld}
Equipped with the terminology and tools in Section \ref{sec:background}, we proceed to understand the structure of heavy-tailed random walks in this section.  We look at heavy-tailed random walks as elements of $\D$. The key result for hidden LDPs for heavy-tailed random walks is in Theorem \ref{thm_large_dev}.

\subsection{Bounds on sums of random variables}
For random variables $Z_{1}, \ldots, Z_{n}$, denote their sum by $S_{n}= Z_{1}+\ldots+Z_{n}$. Here $S_{n}$ denotes the $n$-th step of a random walk. One of the key tools to bound movements in the random walk caused by ``small'' realizations will be Bernstein's inequality, see \cite{bennett1962}.
\begin{lemma}[Bernstein's inequality]
\label{lem_bernstein}
Let $Z_1, \dots , Z_n$ be iid bounded random variables with zero mean, variance $\var{Z_1}=\sigma^2$ and  $|Z_1|\leq M$. Then for any $t>0$,
\begin{align*}
	\mathrm{P}(\left|S_{n} \right| \geq t) < 2\exp\left\{- \frac{ t^2}{ 2n\sigma^2 + \frac{2}{3}Mt} \right\}.
\end{align*}
\end{lemma}



We use this exponential bound on the absolute value of the sum to bound the probability of a large deviation in the sum of regularly varying random variables happening due to many variables attaining a small but non-negligible value. This bound, as we see hence,  turns out to be exponential rather than polynomial in the deviation level $\lambda_{n}$. 

\begin{lemma}\label{lem_small_dev}
Let $\{Z_i\}_{i=1}^\infty$ be a sequence of iid random variables with $Z_1 \in \mathcal{RV}_{-\alpha},\; \alpha >0$. In case the expectation exists, we assume it to be zero.
Denote $S_n= \sum_{k=1}^n Z_k$ and let $\lambda_n \in \mathcal{RV}_\rho$ be a regularly varying sequence such that in case
\begin{align*}
	&\var{Z_1}<\infty, \quad \text{we have} \quad \rho > \frac{1}{2}   \\ 
	&\var{Z_1}=\infty, \quad \text{we have} \quad \alpha\rho >1.		   
\end{align*}

Then for any $\delta>0$  and $\varepsilon_0>0$ small enough, there exists a constant $c>0$ such that for large enough $n$,
\begin{align*}
	\P{|S_n| > \delta \lambda_n, |Z_i| \leq \lambda_n^{1-\varepsilon_0},\; \forall\; i\leq n } < 2\exp(-c\lambda_n^{\varepsilon_0} ).
\end{align*}
\end{lemma}

\begin{remark}
When $Z_{1}$ has finite variance, the condition, $\rho > 1/2$, guarantees that $\lambda_n\uparrow\infty$ fast enough such that we avoid the central limit regime. When $Z_{1}$ has infinite variance, then $\alpha\rho>1$ ensures that the probability of at least one of the variables exceeding a large threshold on the scale of $\lambda_n$ still tends to zero.
\end{remark}

\begin{proof}[Proof of Lemma 3.2]
We transform $S_n$ by making the $Z_{i}$'s bounded  (the bound still depending on $n$);  and then apply Lemma \ref{lem_bernstein}  appropriately to obtain the result.
%
%
First observe that given $\delta>0$, for small enough $\varepsilon_{0}$ and large enough $n$,
\begin{align*}
	\P {A_{n}} &:= \P{|S_n| > \delta \lambda_n, |Z_i| \leq \lambda_n^{1-\varepsilon_0} \forall 1\leq i\leq n } \\
	&\leq \P{\left| \sum_{i=1}^n Z_i \ind{|Z_i|\leq \lambda_n^{1-\varepsilon_0}} \right| > \delta\lambda_n} \\
	&\leq \P{\left|  \sum_{i=1}^n \left(Z_i \ind{|Z_i|\leq \lambda_n^{1-\varepsilon_0}} - \bE{Z_i  \ind{|Z_i|\leq \lambda_n^{1-\varepsilon_0}}}\right)\right| > \frac{\delta}{2} \lambda_n },
\end{align*}
where in the last inequality above we used the following bound, valid for large enough $n$ and some constant $c>0$.
\begin{align*}
	\left| \frac{n}{\lambda_n}\bE{Z_1  \ind{|Z_1|\leq \lambda_n^{1-\varepsilon_0}}} \right| &\leq \frac{n}{\lambda_n}\bE{|Z_1|  \ind{|Z_1|\leq \lambda_n^{1-\varepsilon_0}}}   \\
    & \leq \frac{n}{\lambda_n}\bE{|Z_1|  \ind{|Z_1|\leq \lambda_n}}   \\
    &\leq c n \P{|Z_1|>\lambda_n}.
\end{align*}

Now using  Lemma \ref{lem_bernstein} (Bernstein's inequality) to bound the sum of $n$ zero mean random variables bounded in absolute value by $M = 2 \lambda_n^{1-\varepsilon_0}$, we obtain that for large enough $n$
\begin{align*}
	\P{A_{n}} &\leq 2 \exp\left( - \frac{(\frac{\delta}{2}\lambda_n)^2}{2n\var{Z_1  \ind{|Z_1|\leq \lambda_n^{1-\varepsilon_0}}} + \frac{4}{3}\lambda_n^{1-\varepsilon_0}\frac{\delta}{2}\lambda_n } \right)          \\
	&\leq 2\exp\left( -\lambda_n^{\varepsilon_0} \frac{c_1}{c_2 + \beta(n)}\right),
\end{align*}
where $c_1,c_2$ are positive constants and $$\beta(n) =  \frac{ 2n\var{Z_1  \ind{|Z_1|\leq \lambda_n^{1-\varepsilon_0}}} }{\lambda_n^{2-\varepsilon_0}}.$$
Next we show that $\beta(n)\to 0$ as $n\to \infty$ which will imply that for large enough $n$, there is a $\zeta>0$ such that
\[\P{A_{n}} \le 2\exp\left( -\lambda_n^{\varepsilon_0} \frac{c_1}{c_2 + \zeta}\right) =2\exp\left( -c\lambda_n^{\varepsilon_0}\right), \]
where $c=\frac{c_{1}}{c_{2}+\zeta}$, and thus the lemma is proven.

We show $\beta(n) \to 0$ for three different cases as follows.
\begin{enumerate}
	\item If $\alpha \in (0,2)$ (implying infinite variance and $\alpha\rho>1$), using Karamata's theorem \citep[Proposition 1.5.8]{bingham:goldie:teugels:1989}  for small enough $\varepsilon_{0}$, large enough $n$ and constant $C>0$ we have
	\begin{align*}
	 \beta(n) & \leq \frac{2n}{\lambda_{n}^{2-\varepsilon_{0}}} \E\left[ Z_1^{2}  \ind{|Z_1|\leq \lambda_n^{1-\varepsilon_0}}\right] \\
	               & \sim \frac{2n}{\lambda_{n}^{2-\varepsilon_{0}}} \times C \lambda_{n}^{2(1-\varepsilon_{0})} \P{|Z_1|>\lambda_n^{1-\varepsilon_0}}\\ 
	               & \sim 2 C n\lambda_{n}^{-\varepsilon_{0}} \P{|Z_1|>\lambda_n^{1-\varepsilon_0}} \to 0 \quad (n\to \infty).
	 \end{align*}
	\item If $\var{Z_1} < \infty$, then $\beta(n)\leq C\frac{n}{\lambda_n^{2-\varepsilon_0}}$ for some $C>0$ and hence vanishes as $n \to \infty$ for small enough $\varepsilon_0>0$.
	\item If $\alpha =2$ and $\var{Z_1}=\infty$, then the variance is a slowly varying function. Again, for $\varepsilon_0$ small enough $n\lambda_n^{-2+\varepsilon_0} \to 0$ at a polynomial rate and hence $\beta(n)\to 0$.
\end{enumerate}
\end{proof}
%
%
%
%
%
%
%
%
%
%
%

\subsection{Random walks embedded in $\D[0,1]$} We embed the random walk $S_n=Z_{1}+\ldots+Z_{n}$ in $\D=\D[0,1]$ and discuss its large deviations. Let
 $\bZ^{(n)} = (Z_1,\dots,Z_n)$ where $Z_{i}$'s are iid  realizations from $Z_1$. For $t\in [0,1]$ and $k \in \{1,\dots,n\}$ define the functions $\bX_k^{(n)}(t) = Z_k\ind{\frac{k}{n} \leq t}$ on $\D$. Now define $$\bX^{(n)}(t) := \sum_{k=1}^n \bX^{(n)}_k (t) = \sum_{k=1}^{\lfloor nt \rfloor} Z_k =S_{\lfloor nt \rfloor}$$ to be the embedding of the random walk induced by  $Z_1, \ldots, Z_{n}$ into the space $\D$. Moreover let $\Ceq 0 (\lambda) := \{\bz\in\R^{n}: |z_{i}| \le \lambda, \forall i\}$.
The following corollary is a consequence of Lemma \ref{lem_small_dev}.
\begin{corollary}\label{cor_small_dev}
Under the conditions of Lemma \ref{lem_small_dev}, for any $\delta>0$  and $\varepsilon_0>0$ small enough, there exist a constant $c>0$ such that for large enough $n$,
\begin{align*}
	\mathrm{P}\left(\sup_{t\in [0,1]} |\bX^{(n)}| > \delta \lambda_n , \bZ^{(n)} \in \Ceq{0}(\lambda_n^{1-\varepsilon_0})\right) \leq 2\exp(-c\lambda_n^{\varepsilon_0/2})
\end{align*}
\end{corollary}
\begin{proof} Observe that from Lemma \ref{lem_small_dev}, for any $\delta>0$ and $\varepsilon_{0}>0$ small enough and for large enough $n$,
\begin{align*}
	\P{\sup_{t\in [0,1]} |\bX^{(n)}| > \delta \lambda_n , \bZ^{(n)} \in \Ceq{0}(\lambda_n^{1-\varepsilon_0})} & = 
		\P{ \sup_{t\in [0,1]} |S_{\lfloor nt \rfloor}| > \delta \lambda_n , Z_{i}\le \lambda_n^{1-\varepsilon_0},\; \forall  \; i\le n} \\
	& \le \P{\sup_{1\leq k \leq n} |S_k| > \delta \lambda_n , Z_{i}\le \lambda_n^{1-\varepsilon_0},\; \forall  \; i\le n } \\
	 & \leq 2 n\exp(-c\lambda_n^{\varepsilon_0})\\
	 & \leq 2\exp(-c\lambda_n^{\varepsilon_0/2})
\end{align*}
for some constant $c>0$.
\end{proof}

Now we define functions which relate vectors in $\R^{n}$ to c\`adl\`ag step functions in $\D$. For  integers $j\in \N$, denote $U_j^\uparrow := \{\bu \in {[0,1]^j}: \; 0\leq u_1 < \dots < u_j \leq 1 \}$ and define functions 
\begin{align*}
 	&h_j: \Rjzero \times U_j^\uparrow \to \D, \\
	&h_j((\bz,\bu))(t) :=  \sum_{i=1}^j z_i\ind{u_i\leq t}.
\end{align*}

The maps $h_{j}$ allow us to define the collection of functions with exactly $j$ jumps as subsets of $\D$. Hence we define the following classes of c\`adl\`ag functions.
\begin{align*}
	\D_{=0} &:=\{x\in \D: x(t)=0, 0\le t \le 1\} = \{\text{the zero function in [0,1]} \}, \\
	\D_{=j} &:= \left\{x\in \D: x(t) =  h_{j}(\bz,\bu)(t) , 0\le t\le 1,\quad (\bz,\bu)\in \Rjzero \times U_j^\uparrow) \right\}, \\
	\D_{\leq j} &:= \bigcup_{i=0}^j \D_{=i} = \{\text{c\`adl\`ag step functions with $j$ or less jumps}\}.
\end{align*}

\begin{lemma}\label{lem_h_continuous}
The map $h_j: \Rjzero \times U_j^\uparrow \mapsto \D $  is continuous for $j\in\N$.
\end{lemma}
\begin{proof} The proof which is similar to Lemma 5.3 in \cite{LindskogResnickRoy2014:probSurveys}  is skipped here.

\end{proof}

\begin{remark}
The function $h_{j}$ maps points in $\Rjzero \times U_j^\uparrow $ to functions in $\D_{=j}\subset \D$, which are c\`adl\`ag functions in $[0,1]$ with exactly $j$ jumps. Hence for any $F\subset \D$ bounded away from $\D_{\leq (j-1)}$, we have $h_{j}^{-1}(F) = h_{j}^{-1}(F\cap\D_{=j})$.  Hence  $h_{j}\circ h_{j}^{-1} (F) = F\cap \D_{=j}.$ \end{remark}


The following  result extends the large deviation result of \cite{hult2005} in the setting of \cite{LindskogResnickRoy2014:probSurveys} in order to obtain what we think of as \emph{hidden large deviations}. The special case of $j=1$ in Theorem \ref{thm_large_dev} corresponds to  \cite[Theorem 2.1]{hult2005}. The Lebesgue measure (in $\R^{j}$) is denoted $\Leb_{j}$ and $\nu_{\alpha}^{j}$ is the $j$-fold product measure of $\nu_{\alpha}$ (again in $\R^j$) with $\nu_{\alpha}$ as defined in \eqref{def:nualpha}.

%
%
%
%
%
%
%
%
%
%

\begin{theorem}\label{thm_large_dev}
Let $j\geq 1$. Under the conditions of Lemma \ref{lem_small_dev} and subsequent notations,  for $\lambda_{n} \to \infty$ as $n\to \infty$,
\begin{align}\label{eq:large_dev}
	\gamma_n^{(j)} \P{\bX^{(n)}/\lambda_n \in \;\cdot\;} \to (\nu_\alpha^j \times \Leb_{j}) \circ h_j^{-1} (\cdot)
\end{align}
in $\mathbb{M}(\D\backslash\{\D_{\leq (j-1)}\})$ as $n\to \infty$, where   $\gamma_n^{(j)} = \left[ n \P{|Z_1|>\lambda_n} \right]^{-j}$.
\end{theorem}
Note that under the conditions of Lemma \ref{lem_small_dev} we have $\bX^{(n)}/\lambda_n \to 0$ ($0 \in \D$) in probability.
\begin{remark}
In other words, the theorem states that the random element $\bX_n = \bX^{(n)}/\lambda_n$ satisfies a sequence of LDPs on $\{\D_{\leq j} \}_{j\geq 1}$. That is, the large deviations of the random walk concentrate on step functions with an increasing number of steps at increasingly faster rates.
In particular, for polynomially bounded rates $\gamma_n$ large deviations of partial sum processes of iid regularly varying random variables will always concentrate on step functions, among all functions in $\D$.
%
Naturally $\D\backslash\D_{\leq j}$ are not the only  possible spaces to look at; and other types of LDPs might hold for dependent processes; we do not investigate such possibilities here. 
Additionally we do not investigate large deviations on $\D \backslash \bigcup_{j=0}^\infty \D_{=j}$ in this paper. 
\end{remark}

\begin{proof}[Proof of Theorem \ref{thm_large_dev}]
We show convergence in $\MO$ according to Definition \ref{def_conv_MO} for \eqref{eq:large_dev}, starting with the upper bound for closed sets. The idea is to dissect the space $\R^n$, which contains the first $n$ elements of the random walk, into a union of $n$ disjoint sets that define which dimensions are allowed to be ``big''. For any $k=0,1,\ldots, n$, and $\lambda>0$, define
\begin{align*}
	\Ceq k (\lambda) &= \{ \bz\in \R^n: \; |\{i: |z_i| > \lambda \}| = k\}.
\end{align*}
Hence $\Ceq k \subset \R^{n}$ are all points in $\R^{n}$ which have exactly $k$-co-ordinates with absolute value greater than  $\lambda$. Clearly $$\R_{n} = \bigcup_{k=0}^{n} \Ceq k (\lambda).$$

\paragraph{Upper bound} Let $F\subset \D$ be a closed set, bounded away from $\D_{\leq (j-1)}$. Then for small $\epsilon_{0}>0$,
\begin{align*}
	\P{ \bX^{(n)}/\lambda_n \in F} &= \P{ \bX^{(n)}/\lambda_n \in F, \; \bZ^{(n)} \in \bigcup_{k=0}^n \Ceq k(\lambda_n^{1-\varepsilon_0}) } \\
	 &= \sum_{i=0}^{n}  \P{ \bX^{(n)}/\lambda_n \in F, \; \bZ^{(n)} \in  \Ceq k(\lambda_n^{1-\varepsilon_0}) } =: \sum_{i=0}^{n} \P{B_{i}}.
\end{align*}
We show that when multiplied by $\gamma_{n}^{(j)}$, all the probabilities above are negligible except $\P{B_{j}}$.
Now, since $F$ was chosen to be bounded away from $\D_{\leq (j-1)}$, there exists a $\delta_0>0$, such that all elements of $F$ have a minimum distance  $\delta_0$ to step functions with at most $j-1$ jumps. In particular $F$ is bounded away from the zero element in $\D$. 

\paragraph{1. Bounding $\P{B_{0}}$} Using Corollary \ref{cor_small_dev}, we have constants $c_{0}>0$ and $\epsilon_{0}>0$ such that,
\begin{align*}
	\P{B_{0}}&\leq \P{ \sup |\bX^{(n)}(t)/\lambda_n| > \delta_0/2, |Z_i| \leq \lambda_n^{1-\varepsilon_0}} \\
		&\leq 2\exp(- c_{0}\lambda_n^{\varepsilon_0}).
\end{align*}
Hence this term is exponentially bounded and goes to 0 when multiplied by $\gamma_{n}^{(j)}$.
\paragraph{2. Bounding $\P{B_{i}}$ for $1\le i \le j-1$}  For $i\in \bI:=\{1,2,\ldots,n\}$, denote by $$\bK(i)=\{\bk=\{k_1,\ldots,k_{i}\}: 1\le k_1<\ldots<k_i\le n\},$$ all possible  subsets of size $i$ of the index set $\bI$.  We will show that $\P{B_{i}}$ for $1\le i \le j-1$ are also exponentially bounded. With the same $\delta_{0}$ as previously chosen, we have
\begin{align*}
\P{B_{i}}  & =  \P{ \bX^{(n)}/\lambda_n \in F, \; \bZ^{(n)} \in  \Ceq i(\lambda_n^{1-\varepsilon_0}) }\\
                        & = \sum_{\bk \in \bK(i)}  \P{ \bX^{(n)}/\lambda_n \in F, \;  |Z_{l}| > \lambda_n^{1-\varepsilon_0},\;  \forall l \in \bk, \;  |Z_{l}| \le \lambda_n^{1-\varepsilon_0 },\; \forall \l \in \bI \backslash \bk   }\\ 
                         & \le   \sum_{\bk \in \bK(i)}  \Prob \left(\sup_{t \in [0,1]}\left|\bX^{(n)}(t)-\sum_{m= 1}^{i} \bX_{k_{m}}^{(n)}(t) \right| > \lambda_n \frac{\delta_{0}}2, \; \right.\\ 
                         & \quad\quad\quad\quad\quad \quad\quad\quad   |Z_{l}| > \lambda_n^{1-\varepsilon_0},\;  \forall l \in \bk, \;  |Z_{l}| \le \lambda_n^{1-\varepsilon_0 },\; \forall l \in \bI \backslash \bk   \Bigg)\\
                         & \le   \sum_{\bk \in \bK(i)}  \Prob \left(\sup_{t \in [0,1]}\left|\sum_{l\in \bI\backslash\bk} \bX_{l}^{(n)}(t) \right| > \lambda_n \frac{\delta_{0}}2,  |Z_{l}| \le \lambda_n^{1-\varepsilon_0 },\; \forall l \in \bI \backslash \bk   \right)\\
                         & =   \sum_{\bk \in \bK(i)}  \Prob \left(   S_{ (n-i)}^{*}(\bk)  > \lambda_n \frac{\delta_{0}}2,  |Z_{l}| \le \lambda_n^{1-\varepsilon_0 },\; \forall l \in \bI \backslash \bk   \right),                       
\end{align*}
where 
	$$ S_{(n-i)}^{*}(\bk) = \sup_{t\in[0,1]} \left| \sum_{l\in \bI\backslash\bk} \bX^{(n)}_{l}(t)\right| \eqd \sup_{t\in[0,1]} \left| \sum_{l=1}^{n-i} \bX^{(n)}_l (t) \right| = \sup_{1\leq l \leq n-i} |S_l| $$ 
for any $\bk\in \bK(i)$. Since the size of the set $\bK(i)$ is $|\bK(i)| = \binom{n}{i}$, we have
\begin{align*}
\P{B_{i}} & \le \binom ni \Prob \left( \sup_{1\leq j \leq n-i} \left|S_{j}\right| > \lambda_n \frac{\delta_{0}}2,  |Z_{l}| \le \lambda_n^{1-\varepsilon_0 },\; 1\le l \le n-i   \right),\\
                        & \le 2 \binom ni \exp(-c_{i}\lambda_n^{\varepsilon_0/2}),
\end{align*}
for some $c_{i}>0$ according to Corollary \ref{cor_small_dev}. Since our choice of $\gamma_n^{(j)} = \left[ n \P{|Z_1|>\lambda_n} \right]^{-j}$, clearly 
$\gamma_n^{(j)} \sum_{i=1}^{j-1}\P{B_{i}} \to 0,$ as $n\to \infty$.

\paragraph{3. Bounding $\P{B_{i}}$ for $j+1\le i \le n$} 
We bound the quantity $\gamma_n^{(j)} \sum_{i=j+1}^{n}\P{B_{i}}$ together.  We argue that when multiplied with $\gamma_n^{(j)}$, the  probability of events with more than $j$ large jumps is also negligible. Observe that 
\begin{align*}
	\gamma_n^{(j)}  \sum_{i=j+1}^{n}\P{B_{i}} &\leq \gamma_n^{(j)}  \P{ \exists k_1,\dots, k_{j+1} \in \bI : \;\;|Z_{k_i}| >\lambda_n^{1-\varepsilon_0}, \; i=1,\dots,j+1} \\
		&=  \gamma_n^{(j)}  \binom{n}{j+1} \P{|Z_1| > \lambda_n^{1-\varepsilon_0}}^{j+1}\\
		&   \leq c n \frac{\P{|Z_1|> \lambda_n^{1-\varepsilon_0}}^{j+1}}{\P{|Z_1| > \lambda_n}^j} =:f_{n},
\end{align*}
for some $c>0$. Now $f_{n}$ is a regularly varying sequence with parameter  $$r_{0}:= 1 -(j+1)\rho\alpha + \varepsilon_0\rho(j+1)\alpha + j\rho\alpha =  (1-\alpha\rho) + (j+1)\varepsilon_0\rho\alpha,$$ see \cite[Appendix]{FerreiraDeHaan2006:extreme} for details on operations on regular variation.  Since by choice $\alpha\rho>1$, for small enough $\varepsilon_{0}$, we have $r_{0}<0$. Hence $$\gamma_n^{(j)}  \sum_{i=j+1}^{n}\P{B_{i}}  \le f_{n} \to 0$$ as $n \to \infty$.

\paragraph{4. Bounding $\P{B_{j}}$} Finally, we are left with the term $\gamma_n^{(j)} \P{B_{j}} $ which is the non-negligible contributing term for large $n$. For $\delta>0$, let $$F^{\delta} := \{x \in\D: d_{J_{1}}(x,F)\le \delta\} $$ with $\delta$ small enough such that $F^{\delta}$ is still bounded away from $\D_{\le(j-1)}$.
\begin{align*}
\P{B_{j}} & = \P{ \bX^{(n)}/\lambda_n \in F, \; \bZ^{(n)} \in  \Ceq i(\lambda_n^{1-\varepsilon_0}) }\\
                        & = \sum_{\bk\in \bK(j)}  \P{ \bX^{(n)}/\lambda_n \in F, \;  |Z_{l}| > \lambda_n^{1-\varepsilon_0},\;  \forall l \in \bk, \;  |Z_{k}| \le \lambda_n^{1-\varepsilon_0 },\; \forall l \in \bI \backslash \bk   }\\ 
                            & \le   \sum_{\bk \in \bK(j)}  \text{P} \left(\sup_{t \in [0,1]}\left|\bX^{(n)}(t)-\sum_{m=1}^{j} \bX_{k_{m}}^{(n)}(t) \right| \le \lambda_n \delta, \bX^{(n)}/\lambda_n \in F ,\; \right.\\ 
                         & \quad\quad\quad\quad\quad \quad\quad\quad \quad \quad   |Z_{l}| > \lambda_n^{1-\varepsilon_0},\;  \forall l \in \bk, \;  |Z_{l}| \le \lambda_n^{1-\varepsilon_0 },\; \forall l \in \bI \backslash \bk   \Bigg)\\
                                & \quad \quad +   \sum_{\bk \in \bK(j)}  \text{P} \left(\sup_{t \in [0,1]}\left|\bX^{(n)}(t)-\sum_{m=1}^{j} \bX_{k_{m}}^{(n)}(t) \right| > \lambda_n \delta, \bX^{(n)}/\lambda_n \in F ,\; \right.\\ 
                          & \quad\quad\quad\quad\quad \quad\quad\quad \quad \quad \quad \quad   |Z_{l}| > \lambda_n^{1-\varepsilon_0},\;  \forall l \in \bk, \;  |Z_{l}| \le \lambda_n^{1-\varepsilon_0 },\; \forall l \in \bI \backslash \bk   \Bigg)\\
	&\leq \sum_{\bk \in \bK(j)}  \P{ \sum_{m=1}^{j} \bX^{(n)}_{k_{m}}/\lambda_n \in F^\delta }\\ 
	& \quad \quad + \binom{n}{j} \P{\sup_{t\in [0,1]} \left| \sum_{l \in \bI\backslash \bk} \bX^{(n)}_{l}\right|> \lambda_n\delta,  |Z_{l}| \le \lambda_n^{1-\varepsilon_0 },\; \forall l \in \bI \backslash \bk } \\
	&= P_{j,1}^{(n)} + P_{j,2}^{(n)}.
\end{align*}
Now, using Lemma \ref{lem_small_dev}, and arguments similar to the one for bounding $\P{B_{i}}$ for $1\le i \le j-1$, we can check that the quantity $P_{j,2}^{(n)}$ is negligible at rate $\gamma_n^{(j)}$ and hence $\gamma_n^{(j)}P_{j,2}^{(n)} \to 0$ as $n\to \infty$. \\
In the remaining term $P_{j,1}^{(n)}$, we use the inverse of the map $h_j$ to measure the probability. 
For any closed set $F^{*} \subset \D$, define 
\begin{align*}
	(J(F^*),T(F^*)) :=h_{j}^{-1}(F^{*}\cap\D_{=j}) \subset \Rjzero \times U_j^\uparrow
\end{align*}
to be the pre-image of $F^* \cap \D_{=j}$ under the map $h_{j}$ broken into the \emph{jump} part and the \emph{time} part. 
Also note that due to the continuity of $h_j$ (Lemma \ref{lem_h_continuous}), since $F^{\delta}$  is closed, the pre-image of $F^{\delta}$ is also closed.  
Clearly, $\sum_{m=1}^{j} \bX^{(n)}_{k_{m}}/\lambda_n \in F^\delta$ is equivalent to $\sum_{m=1}^{j} \bX^{(n)}_{k_{m}}/\lambda_n \in F^\delta \cap \D_{=j}$, as $F^\delta$ is bounded away from $\D_{\leq j-1}$. Thus,
\begin{align*}
	P_{j,1}^{(n)} &= \sum_{\bk \in \bK(j)} \P{ \sum_{m=1}^{j} \bX^{(n)}_{k_{m}}/\lambda_n \in F^\delta }\\
	  &= \sum_{\bk \in \bK(j)} \P{ \sum_{m=1}^{j} \bX^{(n)}_{k_{m}}/\lambda_n \in F^\delta \cap\D_{=j} }\\ & = \sum_{\bk \in \bK(j)} \P{ \sum_{m=1}^{j} Z_{k_{m}}\ind{k_{m}/n}(t)/\lambda_n \in F^\delta \cap\D_{=j} }\\
	&= \sum_{{1\leq k_1, \dots,k_j\leq n }} \text{P}\left( \left(\frac{Z_{k_{m}}}{\lambda_{n}}\right)_{1\le m\le j} \in J(F^{\delta})\right) \mathds{1}\left( \left(\frac{k_1}{n},\ldots,\frac{k_j}{n}\right)\in T(F^{\delta}) \right) \\
	& = \text{P} \left(\left(Z_{1},\ldots,Z_{j}\right)/\lambda_{n}  \in J(F^{\delta}) \right)  \sum_{1\leq k_1< \dots< k_j \leq n}   \mathds{1}\left( \left(\frac{k_1}{n},\ldots,\frac{k_j}{n}\right)\in T(F^{\delta}) \right).
\end{align*} 
Note that as $n\to \infty$,
\begin{align}\label{eqn:pj1_{part1}}
 \P{|Z_{1}|>\lambda_{n}}^{-j}\text{P} \left(\left(Z_{1},\ldots,Z_{j}\right)/\lambda_{n}  \in J(F^{\delta}) \right)  \to \nu_{\alpha}^{j}(J(F^{\delta})).
\end{align}
Similarly, for $T(F^\delta)$, we obtain for $n\to \infty$,
\begin{align}\label{eqn:pj1_{part2}}
n^{-j} \sum_{1\leq k_1< \dots< k_j \leq n}   \mathds{1}\left( \left(\frac{k_1}{n},\ldots,\frac{k_j}{n}\right)\in T(F^{\delta}) \right) \to Leb_{j}(T(F^{\delta})).
\end{align}
%
%
Hence using \eqref{eqn:pj1_{part1}} and \eqref{eqn:pj1_{part2}}, we have as $n\to \infty$,
\[\gamma_{n}^{{(j)}} P_{j,1}^{(n)} \to (\nu_{\alpha}^{j}\times\Leb_{j})(J(F^{\delta}),T(F^{\delta})) = (\nu_\alpha^j \times \Leb_{j}) \circ h_j^{-1} (F^\delta). \]
Therefore
\[ \limsup_{n\to\infty} \gamma_{n}^{{(j)}} \P{B_{j}}\le  (\nu_\alpha^j \times \Leb_{j}) \circ h_j^{-1} (F^\delta),  \]
for $\delta>0$.
Summing up  all the bounds we obtained, we have 
$$ \limsup_{n\to\infty} \gamma_n^{(j)} \P{\bX^{(n)}/\lambda_n \in F} \leq (\nu_\alpha^j \times \Leb_{j})(h_j^{-1}(F^\delta)).$$ 
The continuity of $h_j$ ensures $h_j^{-1}(F) = \bigcap_{\delta>0} h_j^{-1}(F^\delta)$ and hence letting $\delta \to 0$ gives us the required upper bound
\[ \limsup_{n\to\infty} \gamma_{n}^{{(j)}} \P{ \bX^{(n)}/\lambda_n \in F}  \le  (\nu_\alpha^j \times \Leb_{j}) \circ h_j^{-1} (F).  \]
\paragraph{Lower bound}
Let $G$ be open and bounded away from $\D_{\leq (j-1)}$. Now define, $G^{-\delta} \subset G$,
\[G^{{-\delta}} = \{f\in G: 
	\; d_{J_{1}}(f,g) < \delta \text{ implies } g\in G\}. \]Choose $\delta$ small enough such that $G^{-\delta}$ is non-empty. It is still open and bounded away from $\D_{\leq (j-1)}$. Searching for a lower bound, we shrink the set $G$ to its bare minimum,
\begin{align*}
	\P{\bX^{(n)}/\lambda_n \in G} &\geq \sum_{1\leq k_1< \dots< k_j \leq n} \P{\sum_{i=1}^j \bX_{k_i}^{(n)}/\lambda_n \in G^{-\delta}, \sup |\bX^{(n)}-\sum_{i=1}^j \bX_{k_i}^{(n)}| < \lambda_n\delta } \\
	&= \sum_{1\leq k_1< \dots< k_j \leq n} \P{\sum_{i=1}^j \bX_{k_i}^{(n)}/\lambda_n \in G^{-\delta}} \P{\sup |\bX^{(n)}-\sum_{i=1}^j \bX_{k_i}^{(n)}| < \lambda_n\delta } \\
\end{align*}
The second factor converges to one since $S_n/\lambda_n \to 0$ in probability as $n\to \infty$.
For the first factor, we proceed in the same fashion as for $P_{j,1}^{(n)}$ above.

\end{proof}
\begin{remark}
Note that instead of our functions being in $\D[0,1]$, we can easily extend Theorem $\ref{thm_large_dev}$  to c\`adl\`ag functions in  $\D_{M}=\D[0,M]$ for some number $M>0$, with minor modifications to the proof. Hence all the results obtained in this section hold if we amend the definitions of the spaces $\D,\D_{=j},\D_{\le j}$ accordingly. Without loss of generality we refer to these results as if they hold for $\D_{M}$ and its appropriate subsets from now on.
\end{remark}

\subsection{Random walks with a constant drift}

The conclusion in Theorem \ref{thm_large_dev}  assumes that the random variables are centred. In case $\bE{Z_1} \neq 0$, we can use the theorem to infer information about the deviations from the mean for a process created with iid variables $Z_{1}^{*} = Z_{1}- \bE{Z_1}$. Nevertheless, if we assume $\alpha>1$ and $\bE{Z_1} \neq 0$ we may look at the random walk with drift. By setting $\lambda_n = n$ we are able to  preserve the drift in the limit. Theorem \ref{thm_large_dev}  can be modified to incorporate a drift term; for this we require two further lemmas as given below. Recall that $e$ denotes the identity function on the respective domain $[0,M]$.
\begin{lemma}\label{lem_add_cont_fun_is_cont}
	Let $f : [0,M] \to \R$ be continuous. Then the map 	$\phi_f: \D_M \to \D_M $ 
\begin{align*}
            			\phi_{f}: & \; x \mapsto x+f
			    \end{align*}
    is continuous in the $J_1$-topology.
\end{lemma}

\begin{proof}
	Let $\varepsilon>0$. Suppose $x,y \in \D_M$ and $d_{J_{1}}(x,y)<\delta$. Then 
    \begin{align*}
    	d_{J_1}(\phi_f(x), \phi_f(y)) &= \inf_\lambda \| x + f - (y + f) \circ \lambda \| \vee \| e- \lambda \| \\ 
        	&\leq \inf_\lambda \left( \| x - y \circ \lambda \| \vee \|e-\lambda\| + \|f-f\circ \lambda\| \vee \| e-\lambda \| \right) \\ 
            &\leq 2 \delta + \| f - f\circ \lambda_{\mathrm{min}} \| \vee 2\delta ,
    \end{align*}
    where $\lambda_{\min}$ denotes a time-shift close to the infimum of the distance of $x$ and $y$. We may bound the fluctuation of $f$ by the modulus of continuity $w_f(\delta)= \sup_{|s-t| \leq \delta } |f(s) - f(t)|$ which tends to zero as $\delta \to 0$ to obtain
    \begin{align*}
    	d_{J_1}(\phi_f(x), \phi_f(y)) &\leq 4\delta + w_f(2\delta).
    \end{align*}
\end{proof}

\begin{lemma}\label{lem_add_const_to_zero_is_ok}
	Let $\{ \mu_n\}_{n\geq 1}$ be a sequence of finite measures on $\mathbb S$ and $\mu_0 \in \M(\mathbb S \backslash \mathbb C)$. Suppose $\mu_n \to \mu_0$ in $\MO$ as $n \to \infty$. 
	Additionally assume there is an addition operation such that $(\mathbb S,+)$ forms a group. Let $\{ y_n \}_{n\geq 1} \subset \mathbb S$ be a sequence with $y_n \to 0$. For $y\in \mathbb S$ denote $s_y: A \mapsto A-y,\; A \in \mathcal B(\mathbb S)$ the map shifting sets by an element $y$. Then  
    \begin{align*}
    	\mu_n \circ s_{y_n} \to \mu_0 , \quad n \to \infty 
    \end{align*}
    in $\MO$.
\end{lemma}

\begin{proof}
	Let $F \subseteq \mathbb S$ be closed and bounded away from $\mathbb C$. Let $\delta >0$ such that $F^\delta$ is still bounded away form $\mathbb C$. For $n$ large enough $d(y_n,0)<\delta$ (where $d$ denotes the metric on $\mathbb S$) and hence
	\begin{align*}
		\mu_n \circ s_{y_n} (F) &\leq \mu_n(F^\delta).
    \end{align*}
    A similar argument holds for open sets $G$ bounded away form $\mathbb C$. Letting $\delta \to 0$ proves the result.
\end{proof}

\begin{corollary}[Corollary to Theorem \ref{thm_large_dev}]\label{cor_hrv_drift}
	Let $\{Z_i\}_{i=1}^\infty$ be a sequence of iid random variables with $Z_1 \in \mathcal{RV}_{-\alpha}, \alpha>1$. Denote $m = \bE{Z_1}$ and define 
\begin{align*}
		&h_j^m: \left(\R\backslash\{0\}\right)^j \times U_j^\uparrow \to \D,\\
		&h_j^m((\bz,\bu))(t) :=  \sum_{i=1}^j z_i\ind{u_i\leq t} + mt,
\end{align*}
and correspondingly $\D_{=j}^m := h_j^m (\R^j\backslash\{0\} \times U_j^\uparrow)$. Then, as $n \to \infty$,
\begin{align*}
	\gamma_n^{(j)} \P{ \bX^{(n)}/n \in  \cdot} \to (\nu_\alpha^j \times \Leb_{j}) \circ \left(h_j^m\right)^{-1} (\cdot), 
\end{align*}
in $\mathbb{M}(\D\backslash \D_{\leq (j-1)}^m)$.
\end{corollary}

\begin{remark}
	The space $\D_{=j}^m$ is defined as the space of step functions with exactly $j$ discontinuities and a constant drift term ``$mt$''. In particular, $\D_{=0}^m = \{ x(t) = mt, \; t \in [0,M]\}$. Theorem \ref{thm_large_dev} allowed for scalings $\lambda_n$ that are growing fast enough such that $\bX^{(n)}/\lambda_n$ stays close to zero for large $n$. Note that in Corollary \ref{cor_hrv_drift} we restrict to $\alpha>1$ and specify $\lambda_n = n$ to preserve the drift term. Necessarily, we observe for sets $A$ bounded away from $\D_{=0}^m$ that $\P{\bX^{(n)}/n \in A} \to 0$ as $n \to \infty$. Hence we examine (a sequence of) large deviation principles on $\D\backslash \D_{\leq j-1}^m$.
\end{remark}

The proof to the corollary is an application of Lemmas  \ref{lem_add_cont_fun_is_cont} and \ref{lem_add_const_to_zero_is_ok}, and the continuous mapping argument of Theorem \ref{thm:cont_map}.

\begin{proof}
By Theorem \ref{thm_large_dev} we have
\begin{align*}
	\gamma_n^{(j)} \P{(\bX^{(n)} - \lfloor ne \rfloor m )/n \in \cdot} \to (\nu_\alpha^j \times \Leb_{j} ) \circ h_j^{-1}(\cdot), \quad n\to \infty,
\end{align*}
in $\mathbb{M}(\D\backslash \D_{\leq (j-1)})$.
Continuous mapping yields
\begin{align*}
	\gamma_n^{(j)} \P{(\bX^{(n)} - \lfloor ne \rfloor m )/n  + me \in \cdot} \to (\nu_\alpha^j \times \Leb_{j} ) \circ (h_j^m)^{-1}(\cdot), \quad n\to \infty,
\end{align*}
in $\mathbb{M}(\D\backslash \D_{\leq (j-1)}^m)$ by virtue of Lemma \ref{lem_add_cont_fun_is_cont}. The result then follows by Lemma \ref{lem_add_const_to_zero_is_ok}.
\end{proof}

\section{Application to finite buffer queues}
\label{sec:queues}
In this section we apply the results of Theorem \ref{thm_large_dev} to the modified Lindeley recursion; see \eqref{eqn:queue:recursion} below. This formula is usually interpreted as describing the evolution of the queue length in a queue with finite buffer. First we derive large deviation principles for what we call long intense periods, defined as the maximum time a queue-size process spends continuously above a certain threshold. These LDPs are derived in a limit where both, the threshold level and the buffer size, approach infinity while the arrival process is sped up appropriately. Second, we present a simulation study which combines two of the derived LDPs to provide a simple analytical approximation and explanation for the empirical distribution of extremal lengths of long intense periods.
\subsection{Queueing processes}
\label{subsec:queueing:processes}
We study recursions of the form 
\begin{align}\label{eqn:queue:recursion}
	Q_n^K =  \min\{ \max\{Q_{n-1}^K + A_n - C_n,0\},K\},
\end{align}
with $Q_0 \geq 0, n \geq 1$ where $\{A_n\}_{n \geq 1}$ and $\{C_n\}_{n\geq 1}$ are two sequences of iid non-negative random variables. This is the modified version of Lindley's recursion \citep{lindley1952} to accommodate queues with finite buffers of size $K$. The recursion in \eqref{eqn:queue:recursion} can be interpreted in many different ways. For example in the context of network traffic, $A_n$ may be interpreted as the number of packets arriving in the time interval $C_n-C_{n-1}$, whereas $Q_{n-1}$ describes the amount of work previously in the buffer of a single server processing work at a fixed rate. Any number of packets arriving at a full buffer are immediately discarded. For example, \cite{jelenkovic1999:lossSubexp} studies \eqref{eqn:queue:recursion} under the assumption that $\int_0^x \P{A_1>z}\diff z/ \E{A_1}$ follows a subexponential distribution to conclude that the stationary loss rate is essentially due to one large observation when the buffer size approaches infinity. That is
\begin{align*}
	\bE{(Q_n^K +A_{n+1} - C_{n+1} - B) \vee 0 } = \bE{(A-K) \vee 0}(1 + o(1)), \quad K\to \infty.
\end{align*}

Sample path LDPs for queueing processes with both infinite and finite buffers are studied in \cite{ganesh2004:bigQ} mostly under the assumption that the moment generating function exists. We work with regularly varying random variables $A_1 \in \RV{-\alpha}, \alpha>0$ throughout which do not satisfy this assumption. 

To study the queueing recursion \eqref{eqn:queue:recursion} we follow the continuous mapping approach. First we define a suitable embedding of the sequences $\{A_n\}$ and $\{C_n\}$ in the space $\D_M$ and then employ a continuous map to obtain a process which agrees with the queueing recursion at specified discrete time stamps. See for example \cite{whitt2002:SPL}, \cite{asmussen2003:apq} or \cite{Andersen2015} for more on this approach applied to queueing processes. This map is usually called reflection or Skorohod map. 
We briefly recall the required results.
For a process $x \in \D_M$ with $x(0) = 0$ we call $\{v(t), l(t), u(t)\}$ the solution to the Skorohod problem if $v(t) \in [0,K]$ and
\begin{align*}
	v(t) = x(t) + l(t) - u(t), \;\;
	\int_0^\infty v(t) \diff l(t) = 0, \;\;
	\int_0^\infty (K-v(t)) \diff u(t) = 0,
\end{align*}
and both $l,u$ are non-negative non-decreasing functions.
Denote $\psi_0^K:\; \D_M \to \D_M$ the reflection map on the interval $[0,K]$ as
\begin{align}\label{def_eqn_reflectionmapping}
	\psi_0^K : x \mapsto v, 
\end{align}
where $v$ denotes the resulting regulated process of the Skorohod problem. 
%
This map is (Lipschitz-) continuous on $\D_M$ equipped with the $J_1$ metric, see e.g. Lemma 4.6 of \cite{Andersen2015}.

To facilitate the discussion let $C_n = c, n\geq 1$ for some $c>0$ and denote 
\begin{align*}
	\bA(t) := \sum_{i=1}^\infty (A_i-c)  \ind{t\geq i},  \; t \geq 0
\end{align*} 
the embedding of the random walk induced by $A_n - c$ in $\D_M$. Then 
\begin{align}\label{def:QK}
	\bQ^K := \psi_0^K(\bA)
\end{align}
is an embedding of $Q_n^K$ into $\D_M$ satisfying $\bQ^K(t) = Q_t^K$ for $t \in \N_0$. Consequently we call $\bQ^K$ a queueing process with buffer $K$.

\begin{remark}
	We could also work with other embeddings as for example $$\bB(t):= \sum_{n=1}^\infty A_n\ind{t \geq n} - ct, \; t \geq 0$$ and define $\bQ_B^K := \psi_0^K(\bB)$ to allow for more nuanced interpretations of the queueing process $\bQ$. But since this work focuses on scaled versions of the queueing process with both time $t$ and space $\bQ^K(t)$ scaled appropriately, the exact form of the interpolation is mostly irrelevant for the limit.
\end{remark}

\subsection{Long intense periods}
\label{subsec:long:intense:periods}

We adopt the position that a queueing process with the queue size close to the buffer $K$ corresponds to an undesirable state. In such a state the service quality (of which $\bQ^K(t)$ is a proxy) is perceived as suboptimal. In the following we introduce and study the longest period an observed queueing process spends above a certain threshold $\theta K$ during the observation horizon $[0,M]$. We call such intervals long intense periods and investigate their length. 
\begin{definition}[Long intense period]
	For a c\`adl\`ag function $x \in \D_M$ and a fixed level $\eta \in \R_+$ we define
	\begin{align}\label{def:lip}
		\begin{split}
		& L^\eta: \D_M \to \R_+ \\
	 	& \qquad x \mapsto \sup_{0\leq s < t \leq M} \{ t-s: x(u) >\eta\; \forall u \in (s,t) \}. 
	 	\end{split}
	\end{align}
	For a queueing process $\bQ$ with buffer $K$ we call $L^{\theta K}(\bQ^K)$ the length of the intense period at level $\theta \in (0,1)$. 
\end{definition}

\begin{example} \label{ex:longintense}
 How useful is it to calculate large deviations for long intense periods in queues? Can we gain more insight into waiting times in queues with this information To illustrate the applicability of such results, we use a simulation study which investigates the distribution of long intense periods for large threshold levels $\theta \in (0,1)$.

\begin{figure}[h]
 	\centering
 	\includegraphics[scale=0.5]{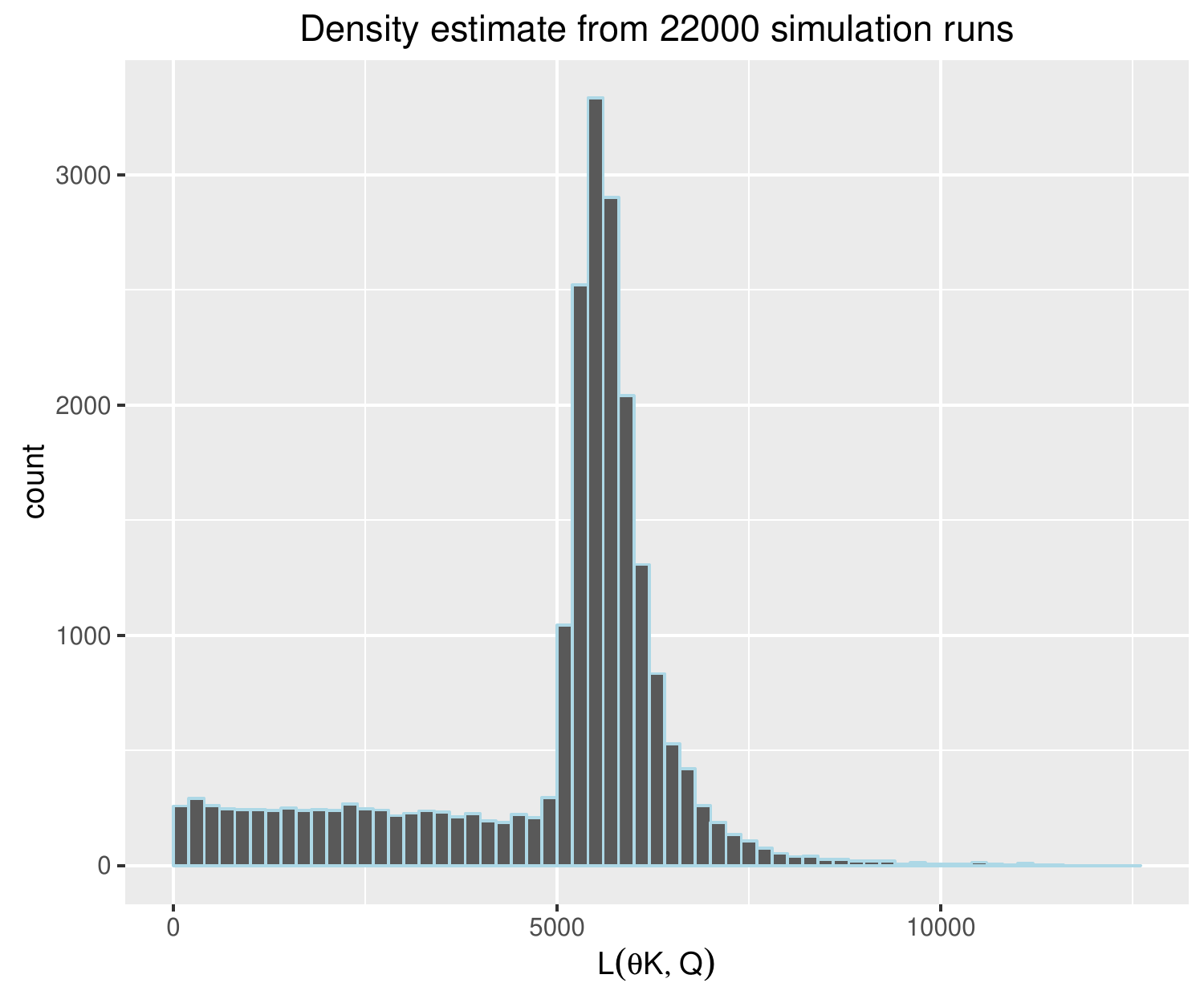}
 	\caption{Histogram for $L^{\theta K}(\bQ^K)| L^{\theta K}(\bQ^K)>0$ generated by 22000 realizations of a queueing process with $\bE{B} = \rho=0.5$, 50000 arrivals at integer time points and maximum capacity $K=20000$. The critical level was set to $\theta = 0.85$. The service time distribution follows an exact power law with $-\alpha = -1.44$.}
 	\label{fig:empty_hist_L}
\end{figure}
The object of our simulation study is a queueing process $\bQ^K(t)$ with $N=50000$ arrival variables following a power law distribution with tail index $\alpha=1.44$ and expectation $m = 0.5$. The queue has a finite buffer $K=20000$ and any additional service requirements will be lost. We assume the server works at a fixed rate $c=1$ with $A_i$ describing the amount of service requirements arriving in one unit of time. 
We study long intense periods above the level $\theta=0.85$, that is $\bQ^K(t)> 17000$ is considered intense. The queueing process is observed on $[0,M]$ with $M=N$.
With this example we treat service time distributions that still have finite means but infinite variance. The particular $\alpha$ value corresponds to the tail parameter of file sizes in Internet traffic reported in \cite{jelenkovic2003}.
Specifically we consider the arrival distribution 
\begin{align*}
	\P{A_1 > z} = \left( \frac{z}{(\alpha-1) m} + 1 \right)^{-\alpha} , \; z>0.
\end{align*}
Trivially $\P{A_1 >z} \in \RV{-\alpha}$.

Figure \ref{fig:empty_hist_L} contains a histogram of the realized lengths of the long intense periods in queueing processes with the parameters above. It is based on 22000 observations which exhibit a strictly positive long intense period. That is, Figure \ref{fig:empty_hist_L} shows a histogram of $L^{\theta K}(\bQ)|L^{\theta K}(\bQ) >0$. We would like to understand the shape of the histogram that we observe here; why is there a peak in the middle and a  decay afterwards? We revisit the histogram at the end of this section, accompanied by an explanation for its shape, based on (hidden) large deviations.
\end{example}

\subsection{Large deviations for long intense periods}
%
We work out the corresponding sequence of large deviation principles for long intense periods of queueing processes.
\begin{theorem}\label{thm:longintense}
	Let $A_i, i\geq 1$ be a sequence of iid non-negative regularly varying random variables with $A_1 \in \RV{-\alpha}$, $\alpha>1$. Assume $c>m:=\bE{A_1}$ and define the queueing process $\bQ^{K,(n)}(t):= \bQ^{nK}(nt)/n, \; t\in [0,M], \; n \geq 1$ with $\bQ^K$ defined as in \eqref{def:QK}. Denote $\kappa :=\frac{1-\theta}{c-m} K$.
	
	The intense periods $L_n:= L^{\theta K} (\bQ^{K,(n)})$ of the queueing process $\bQ^{K,(n)}$ observed on $[0,M]$ satisfy a sequence of large deviation principles on $[0,M] \backslash [0,(j-1)\kappa ]$ with the limit measure $\mu^{(j)}_L$ concentrating it's mass on $\left( (j-1)\kappa, j\kappa \right]$. Specifically,
	\begin{align*}
		L_n \in \LD{ \gamma_n^{(j)}, \mu^{(j)}_L, [0,M] \backslash [0,(j-1)\kappa]}, \quad 1 \leq j \leq \left\lfloor \frac{M}{\kappa} \right \rfloor,
	\end{align*}
	where the limit measure is given by 
	\begin{align*}
		\mu^{(j)}_L = (\nu_\alpha^j \times \Leb_{j}) \circ \left(h^{m-c}_j\right)^{-1} \circ \left( \psi_0^K \right)^{-1} \circ \left( L^{\theta K}\right)^{-1}.
	\end{align*}
\end{theorem}
\begin{remark}
The assumption $c>\bE{A_1}$ ensures that the process will drift in the negative direction on average, such that the process being close to its buffer is actually a rare event. At the first level for $j=1$ the theorem states that the long intense periods of a queueing process with buffer $K$ and negative drift may be approximated by summing over all one-jump functions that contain a jump of size at least $\theta K$. Since the measure concentrates on one-jump functions, the maximum attainable long intense period is attained by a single jump that exceeds the buffer limit $K$, with the process drifting in negative direction at rate $m-c$ afterwards. Thus, latest at time $\kappa$ after the jump the process will leave the intense region, no matter the size of the jump. 
\end{remark}
%
%
%
\begin{remark}
Measuring the longest connected interval of time spent above a certain threshold is not a continuous operation for c\`adl\`ag processes. For example consider for $M>2$ the function $x\in \D_M$
\begin{align*}
	x(t) := \begin{cases}
		1-t &\text{if } t \in [0,1) \\
		2-t &\text{if } t \in [1,M].	
	\end{cases}
\end{align*}
Adding a small constant via $\phi_c(x)(t):= x(t) +c$ we obtain for $c<0$: $L^0(\phi_c(x)) = (1-|c|)\wedge 0$ but $L^0(x) = L^0(\phi_0(x))=2$, while at the same time $\phi_c(x) \to x$ as $c\to 0$. Consequently $L$ is not continuous. 
Nevertheless $L^{\theta K}$ is continuous almost everywhere with respect to the limit measure $\mu = \nu_\alpha \times \Leb_{1} \circ \left(h_1^m\right)^{-1} \circ \left( \psi_0^K \right) ^{-1}$ on $\D_M$ as the only way to obtain a discontinuity is through the jump at the end of the long intense interval. But the jump position is uniformly distributed hence the measure of that set is zero. Additionally, for our purposes, there is no need to consider functions outside the support of $\mu$.
%
\end{remark}
We need the following lemma to prove Theorem \ref{thm:longintense}.
\begin{lemma}\label{lem:L_cont}
Let $j \in \N$. Denote
\begin{align*}
	E_j&:= \left\{x \in \D_M: \begin{array}{l}
	\left[ x(t)=0\right]  \text{ OR } \left[x(t+s) = x(t) - s(c-m) \text{, $|s|$ small enough}\right] \\
	\text{for all but $j$ points $t$. Additionally } x(t) \geq x(t-) \forall t\in [0,M].
	\end{array}	\right\},\\
	D_j &:= \left\{\begin{array}{l} x \in E_j: \exists t \in \{\text{discontinuity points of $x$}\} \\
		\text{such that}\; x(t-) = \theta K \text{ OR } x(t)=\theta K \text{ OR } t \in \{0,M\}
		\end{array} \right\}.
\end{align*}

Then $L^{\theta K}$ is continuous on $E_j\backslash D_j$.
\end{lemma}
\begin{remark}
Note that $E_j$ contains all c\`adl\`ag functions which contain exactly $j$ positive jumps and decrease at rate $c-m$ otherwise, regulated to take values in $[0,K]$. The set $D_j$ further restricts to those functions whose jumps are bounded away from the critical level $\theta K$.
\end{remark}
\begin{proof}[Proof of Lemma \ref{lem:L_cont}]

To show the claim we need to introduce additional machinery. Namely we define an intense period as a period during which the function $x \in E_j$ stays continuously above the critical level $\theta K$ and enumerate all such periods. Subsequently we show that the length of each such period cannot change much in case $x$ is not perturbed too much.
Denote 
\begin{align*}
	_sL(l,x,v),\ _tL(l,x,v)&: \;\R_+ \times \D_M \times [0,M] \to \R_+ \\
	_sL(l,x,v) &:= \inf\{u \in (v,M]:\; x(u) > l \}, \\
	_tL(l,x,v) &:= \inf\{u \in (\ _sL(l,x,v),M]:\; x(u)<l\; \mathrm{OR} \; u=M \}.
\end{align*}
We assume - as is usually the case - that $\inf \emptyset = \infty$. Next we recursively record the start and end times of what we call intense periods, starting at zero.
\begin{align*}
	s_1,t_1 &:=\ _sL(l,x,0),\ _tL(l,x,0), \\
	s_i,t_i &:=\ _sL(l,x,s_{i-1}),\ _tL(l,x, s_{i-1}), \;\; i\geq 2, \\
	n_x &:= \max \{i: s_i < \infty\}
\end{align*}
In case the tuple $s_i,t_i$ is finite - either both or none are - we call $t_i-s_i$ the length of the $i^{th}$ intense period of $x$. Note that for $x \in E_j$ there are exactly $n_x$ intense periods, with $0\leq n_x \leq j$. The length of the longest of these corresponds to what we defined above in \eqref{def:lip} as the length of the long intense period of $x$. The enumeration of intense periods is well-defined for any c\`adl\`ag function on $[0,M]$, although their number need no longer be finite, nor must the linear enumeration from zero capture all instances of the function exceeding a level $l$.

Let $x \in E_j \backslash D_j$. Then the set of time points at which $x$ is above the critical level can be partitioned as
\begin{align*}
	\{ u:\; x(u) > \theta K\} = \bigcup_{i=1}^{n_x} [s_i,t_i).
\end{align*}
Since all jump discontinuities of $x$ have values bounded away from the critical level, all of the intervals $[s_i,t_i)$ and $[t_i,s_{i+1})$ are of positive length. Moreover, denoting
\begin{align*}
	\Delta^{\theta K}_J(x) := \min \{ |x(u^-) - \theta K|\wedge |x(u) - \theta K | \wedge u-0 \wedge M-u :\; u \text{ is a discontinuity point of }x\},
\end{align*}
we find that for all $0<\delta < \Delta^{\theta K}_J(x)$, 
\begin{align}\label{eqn:intense_iff}
	x(u) \in (\theta K -\delta, \theta K + \delta) \Leftrightarrow u \in \left(t_i - \frac{\delta}{c-m}, t_i + \frac{\delta}{c-m}\right)\cap [0,M] \text{ for some } 1\leq i \leq n_x.
\end{align}
Next we show that for all $\varepsilon>0$, small enough such that $\varepsilon(c-m)< \Delta^{\theta K}_J(x)$, there exists a $\zeta$ such that whenever $d_{J_1}(x,y) < \zeta$ we have 
\begin{align}\label{eqn:L_continuity_statement}
	\begin{split}
	\begin{array}{lll}
		u \in [s_i+\varepsilon,t_i-\varepsilon) &\Rightarrow &y(u) > \theta K , \\
		u \in [t_i + \varepsilon,s_{i+1}-\varepsilon) &\Rightarrow &y(u) < \theta K.
	\end{array}
	\end{split}
\end{align}
This implies that any $y$ close enough to $x$ has similar intense periods as $x$, ignoring any negligible intense periods of $y$. Hence, $L^{\theta K}$ is continuous at $x$ in $(\D_M,d_{J_1})$. 
We proceed by showing that the above claim holds for $\zeta = \frac{\varepsilon ((c-m)\wedge 1)}{3}$. Then there exists a $\lambda \in \Lambda$ such that 
\begin{align}
	\| x- y\circ \lambda \| &< \frac{\varepsilon ((c-m)\wedge 1)}{2}, \label{eqn:L_bound_vertical}  \\
	\| \lambda - e\| &< \frac{\varepsilon ((c-m)\wedge 1)}{2}. \label{eqn:L_bound_time}
\end{align}
Now \eqref{eqn:L_bound_vertical} combined with \eqref{eqn:intense_iff} (where $\delta = \varepsilon(c-m)/2$) implies
\begin{align}\label{eqn:L_continuity}
\begin{split}
	u \in [s_i , t_i-\varepsilon/2) &\Rightarrow (y \circ \lambda)(u)> \theta K + \frac{\varepsilon (c-m)}{2} - \frac{\varepsilon((c-m)\wedge1)}{2} \geq \theta K, \\
	u \in (t_i-\varepsilon/2,s_{i+1}) &\Rightarrow (y \circ \lambda)(u)< \theta K.
\end{split}
\end{align}
Accounting for the time change introduced through $\lambda$, we infer from \eqref{eqn:L_bound_time} that the last two implications in \eqref{eqn:L_continuity} hold when the two intervals get reduced by a further $\varepsilon/2$ on each side. In turn this proves the statement in \eqref{eqn:L_continuity_statement} and thus continuity of $L^{\theta K}$ on $E_j\backslash D_j$ for all $j\geq 1$. 
\end{proof}
\begin{proof}[Proof of Theorem \ref{thm:longintense}]
We apply the continuous mapping argument in Theorem \ref{thm:cont_map} twice.
First, using the Skorohod map of \eqref{def_eqn_reflectionmapping}, the large deviations result in Corollary \ref{cor_hrv_drift} and continuous mapping yield
\begin{align}\label{eqn:limit_measure_Q}
	\gamma_n^{(j)} \P{Q^{nK}(nt)/n \in \cdot } \to (\nu_\alpha \times \Leb_{j}) \circ \left(h^{m-c}_j\right)^{-1} \circ \left( \psi_0^K \right)^{-1} (\cdot), \quad n \to \infty
\end{align}
in $\M \left( \psi_0^K ( \D_M \backslash (\D_M)^{m-c}_{\leq j-1}) \right)$. This is due to the definition in \eqref{def_eqn_reflectionmapping} satisfying $\psi_0^K (x/n) = \psi_0^{nK}(x)/n$. The queueing map $\psi_0^K$ preserves the number of jumps and hence satisfies the ``bounded away'' condition of Theorem \ref{thm:cont_map}.

It is immediate from the definition of the Skorohod problem that $\psi_0^K(\D_M \backslash (\D_M)^{m-c}_{\leq j-1}) \subseteq E_j$.
Additionally, note that $\mu_Q^{(j)} (D_j) = 0,\; j\geq 1$ as $\left(h^{m-c}_j\right)^{-1} \circ \left( \psi_0^K \right)^{-1}(D_j) \subset \R^{2j}$ is not of full dimension; the condition of having $x(t)\text{ or } x(t^-) = \theta K$ amounts to imposing restrictions linking the time and value of a jump through relations of the form 
\begin{align*}
	x(t_i) - (t_i-t_{i-1}) (c-m) + J(t_i) = \theta K,
\end{align*}
where $J(t_i)$ denotes the size of the $j\thenum$ jump.
Hence Lemma \ref{lem:L_cont} above combined with a second application of Theorem \ref{thm:cont_map} yields the result.

\end{proof}

\subsection{Calculating explicit limit measures}

In the following we compute the limit measures $\mu_{L}^{(1)}$ and $\mu_{L}^{(2)}$. We assume $M> 2\kappa$ throughout. For $j=1$ we obtain
\begin{align*}
	\mu^{(1)}_{L}( (l,\infty) ) = \begin{cases}
			(M-l) ( l(c-m) + \theta K)^{-\alpha} &\text{if } l \in (0,\kappa] \\
			0 &\text{otherwise}.	
	\end{cases}
\end{align*}
In other words, the measure $\mu_{L}^{(1)}$ is the sum of a point mass at $l = \kappa$ with value $K^{-\alpha}(M-\kappa)$ and an absolutely continuous part on $(0,\kappa)$. 
Considering this initial large deviations estimate on its own we would approximate $\Prob ( L^{\theta K}( \bQ^K) > \kappa ) \approx 0$. For any finite buffer non-limit scenario this may be too coarse. A more refined estimate based on hidden large deviations allows for more accuracy. Namely on $[0,M] \backslash [0,\kappa ]$ we have 
\begin{align*}
	\gamma_n^{(2)} \P{ L^{\theta nK}(\bQ^{nK}(nt)) \in \cdot } \to \mu_{L}^{(2)}(\cdot), \quad n \to \infty,
\end{align*}
in $\M( [0,M] \backslash [0,\kappa])$, which concentrates on $(\kappa,2\kappa]$. This again can be explained by the rate $\gamma_n^{(2)}$ only allowing for at most two jumps in the random walk. Any intense period with length $L<\kappa$ is more likely to happen due to one jump hence processes containing only one jump must be excluded in the hidden large deviation principle. Long intense periods with length $L>2\kappa$ are not possible since the maximum length is achieved if the buffer is filled at some initial time $t_0< M-2\kappa$ starting the long intense period and an additional jump at time $t_0 + \kappa$ of size at least $(1-\theta)K$. We compute the limit measure for the events $\{L > l\}, \; l \in (\kappa,2\kappa]$:
\begin{align}
	\begin{split}\label{eqn:two:jumps:L}
	\mu_{L,\theta,K}^{(2)} ( (l,\infty) ) &= \mu^{(2)} ( \{ \text{All two jump functions with $L>l$} \}) \\
		&= (M-l) \int_{\theta K}^\infty \nu_\alpha (\diff j_1) \int_{l-\kappa}^{(K\wedge j_1 - \theta K)/(c-m)} \diff u_2 \int_{l(c-m) - ( K\wedge j_1 - \theta K)}^\infty \nu_\alpha(\diff j_2) \\
		&= (M-l) \int_{\theta K}^\infty \nu_\alpha(\diff j_1) \ind{\frac{K\wedge j_1 - \theta K }{c-m} > l-\kappa} \frac{  \frac{K\wedge j_1 - \theta K }{c-m} - (l-\kappa)}{ \left( l(c-m) - (K \wedge j_1 - \theta K)   \right)^{\alpha} } \\
		&= \frac{M-l}{c-m}  \int_{\theta K + l(c-m) - (1-\theta)K}^K \alpha x^{-\alpha-1} \frac{x - \theta K - l(c-m) + (1-\theta K) }{\left( l(c-m) + \theta K - x\right) ^\alpha } \diff x \\
		&\quad + \frac{M-l}{c-m}K^{-\alpha} \frac{2(1-\theta) K - l(c-m) }{ \left( l(c-m) - (1-\theta) K \right)^\alpha } .
	\end{split}
\end{align}
\begin{remark}
	Further limit measures can be computed but the explicit derivation becomes more cumbersome as the level increases.
\end{remark}	
\subsection{Simulation study - combining the first two LDPs}
\label{sec:simul}
%
%
%
%
%
%
%
%
%
%
%
%
In this section we provide some insights on the practical relevance of hidden large deviations. 
%
%
%
%
We find that in the setting of the simulation study described in Example \ref{ex:longintense} we are able to numerically validate the rate and limit measure of hidden large deviations. 
The previous section established large deviation principles for long intense periods for any interval $[(j-1)\kappa, j\kappa]$ with $j\leq \lfloor M/\kappa \rfloor$, each with its own rate. 
And indeed, for the theory of LDPs we may only treat these limit measures separately due to the different magnitudes of the rates $\gamma_n^{(j)}$. In practice however, for any finite observation period of a queue with finite buffer size, several of the limit measures might be relevant for a single statistic. 
The simulation study will examine the interplay of different rates in a single probability estimate. 
We proceed to construct the two estimates involving the first and second level LDPs separately.

\begin{figure}[h]
	\centering
	\includegraphics[scale=0.5]{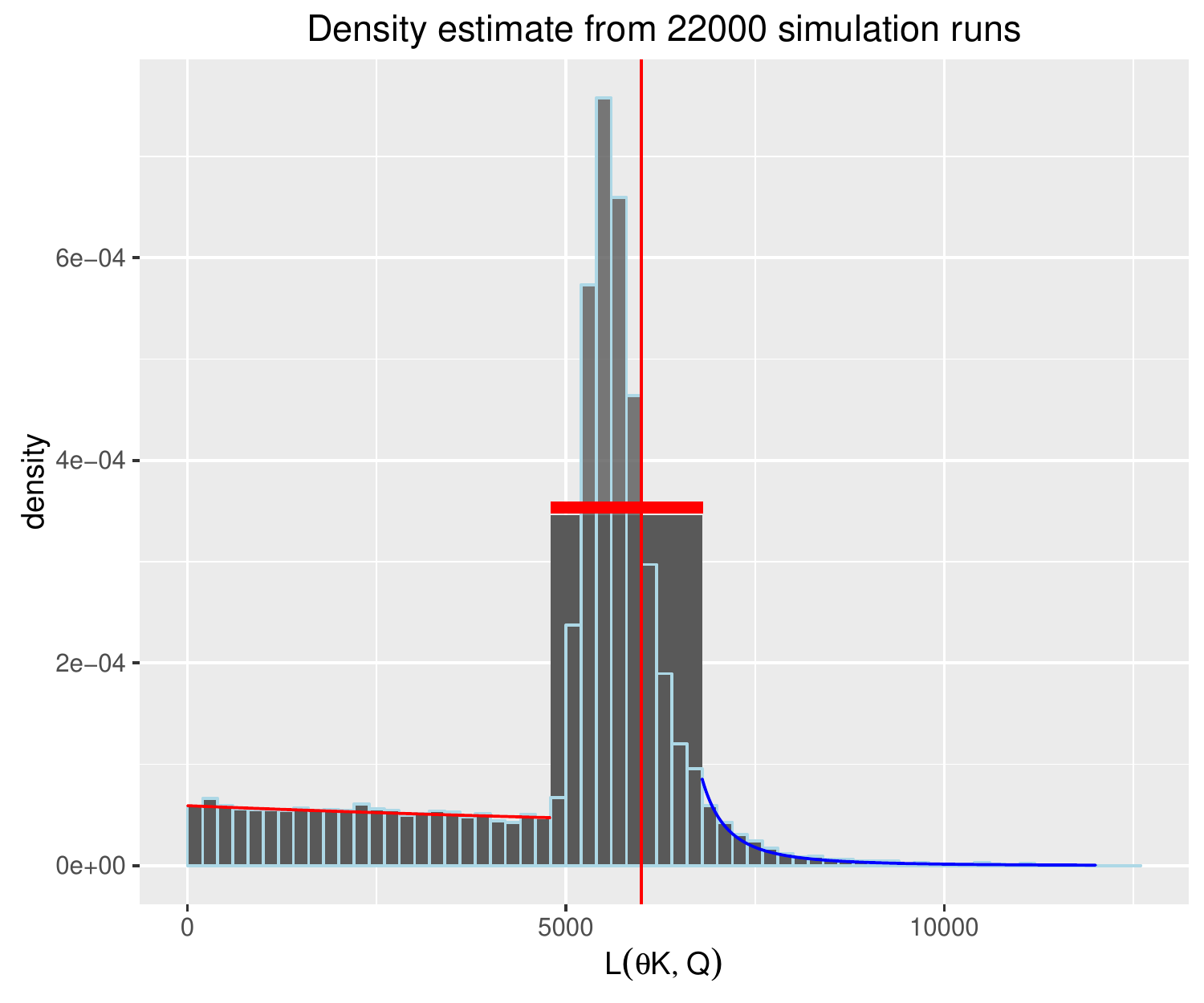}
    \includegraphics[scale=0.5]{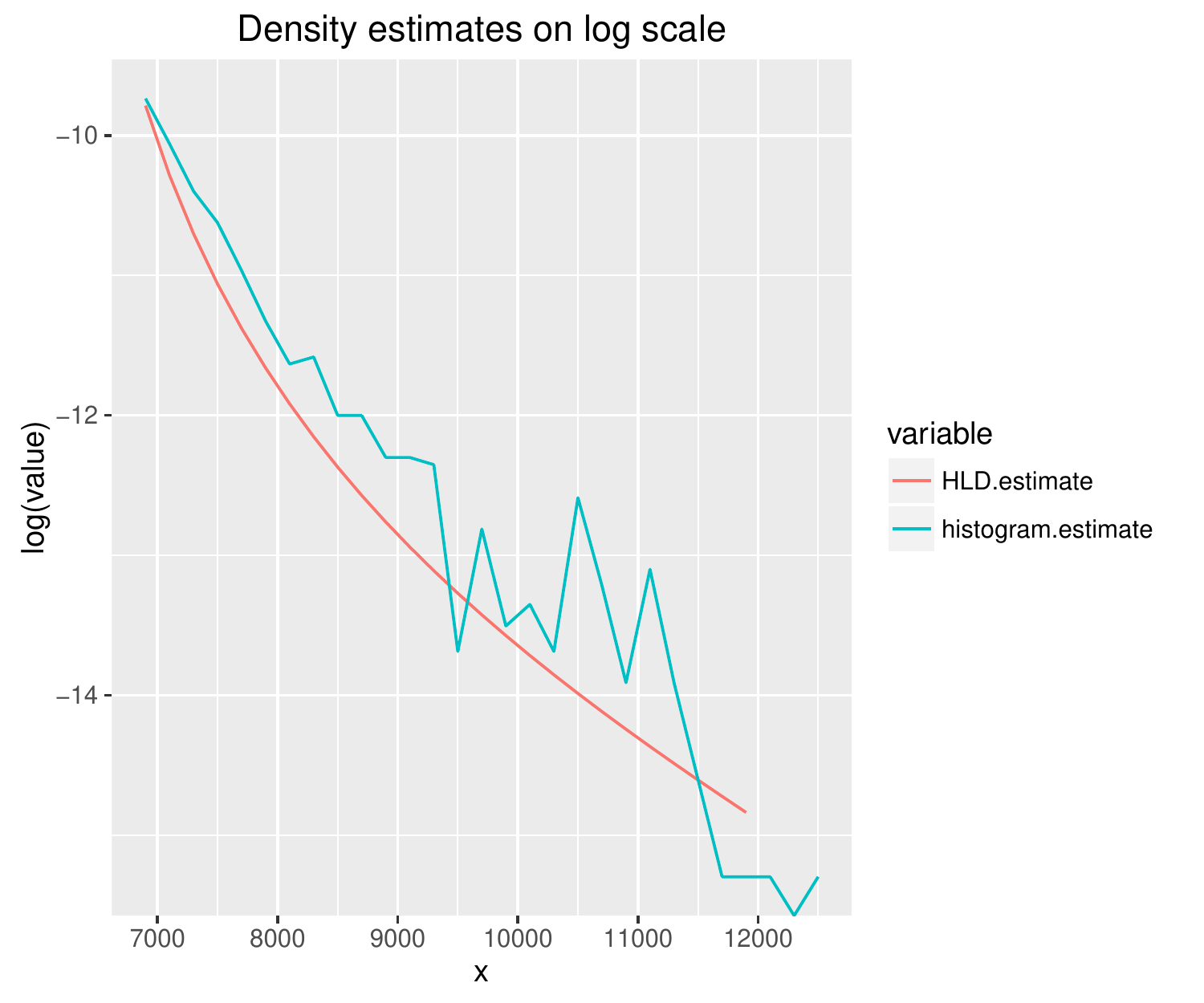}
	\caption{Left: Histogram for $L^{\theta K}(\bQ)| L^{\theta K}(\bQ)>0$ generated by 22000 realizations of a queueing process with $\bE{B} = \rho=0.5$, 50000 arrivals at integer time points and maximum capacity $K=20000$. The critical level was set to $\theta = 0.85$. The red vertical line marks the location of the theoretical point mass at $L = \frac{1-\theta}{c-m} K$. The service time distribution follows an exact power law with $-\alpha = -1.44$. Approximations of the densities with the large deviations estimate (in red) and the hidden large deviations estimate (in blue) are added on top of the histogram.
    Right: Same data as on the left restricted to $L>7000$. Density estimate with HLD compared to empirical values of the histogram on logarithmic scale.}
    \label{fig_hist_and_dens}
\end{figure}

\paragraph{One jump}
According to traditional large deviation estimates for heavy tailed queueing processes with ``large'' buffers, the long intense period will be due to a single jump reaching above the threshold level $\theta K$ and the queue drifting in direction $-(c-m)$ thereafter.
To use Theorem \ref{thm:longintense} we need to choose a queue sequence number $n$. 
We thus obtain the following approximation.
\begin{align*}
	\P{ L^{\theta K}( \bQ^K)> l } &= \P{L^{\theta K/n}(\bQ^{K/n,(n)}) > l/n}    \\
				&= \P{L_n > l/n} \\
				&\approx \frac{1}{\gamma_n^{(1)}} \mu^{(1)}_L((l/n,\infty)) \\		
	&= \begin{cases}
		n^{\alpha} \P{A_1>n} (M-l) (l(c-m) + \theta K)^{-\alpha} & \text{if } l \in (0, \kappa] \\
		0 & \text{otherwise} .
	\end{cases} 
\end{align*}
The point mass at $\kappa$ yields
\begin{align*}
	\P{ L^{\theta K}(\bQ^K)  \in (\kappa - \varepsilon ,\kappa + \varepsilon)} \approx \left(\frac{n}{K}\right)^\alpha \P{A_1 > n} (M-\kappa).
\end{align*}

\paragraph{One or two jumps}
The two jump measure can be approximated in the same fashion as the approximation for one jump above using equation \eqref{eqn:two:jumps:L} instead. To get a single estimate for the distribution of the long intense periods of the queueing process $\bQ^K$ we propose to combine the two estimates into a single approximation.
\begin{align}
	\begin{split}\label{eqn:approx:L}
	\P{ L^{\theta K}( \bQ^K) > l} \approx \begin{cases}
		n\P{A_1>n} \mu_{L}^{(1)}( (l/n,\infty) ) & \text{if }  l \in (0,\kappa] \\
		n^2 \P{A_1>n}^2 \mu^{(2)}_{L}( ( l/n,\infty) ) & \text{if }  l \in ( \kappa, 2 \kappa] \\
		0 & \text{otherwise},
		\end{cases}
	\end{split}
\end{align}
where the buffer size $K$ and observation horizon $M$ are scaled accordingly in the limit measures.

In Figure \ref{fig_hist_and_dens} we plot the same histogram as in Figure \ref{fig:empty_hist_L}, and view it as an estimate of the density 
$$\P{L^{\theta K}( \bQ^K) \in \diff l | L^{\theta K}( \bQ^K)> 0},\; l>0.$$ 
The limit measure $\mu^{(1)}_{L}$ puts zero mass on values beyond the vertical red line which marks the location of the point mass of $\mu^{(1)}_{L}$. Due to $n$ being finite we expect some values immediately to the right of the point mass as caused by only a finite number of random variables approximating the mean rate of decrease for the queue content. Nevertheless, concerning the values on the far right we believe an explanation via Hidden Large deviations (HLD) is best suited for the distribution of $L^{\theta K}( \bQ^K)$.
Hence we add the estimate in \eqref{eqn:approx:L} to the plot. 
To visualize the point mass we fix two $\varepsilon_1,\varepsilon_2>0$ such that the point mass at $\kappa$ gets distributed over the area $(\kappa-\varepsilon_1 , \kappa + \varepsilon_2)$. Outside of this region we approximate the measure with the corresponding densities.
Additionally we provide a plot of the tail of the distribution on a log scale to better visualize the fit for the hidden large deviation estimate. The figures clearly show how our hidden large deviation estimates closely approximate the histogram observed (more clearly in the right plot in Figure \ref{fig_hist_and_dens}.)


%
%
%
%
%
%
%
%
%
%
%
%
%
%
%
%
%
%
\section{Conclusion and further remarks}\label{sec:conclusion}

We provide limit measures for successively rarer large deviations of random walks with regularly varying iid increments. Scaling time and space appropriately we are able to obtain limit measures for large deviations of queueing processes which preserve the drift term in the limit. In the final section we showed that hidden large deviations at the second level, though happening at the squared rate of the first large deviation, are numerically observable. Clearly, our (hidden) large deviation estimate performs quite well to approximate the histogram - even at the tail (on a log scale.)  


For future directions of study, one can explore large deviations on a space $\D\backslash \bigcup_{j=1}^\infty \D_{=j}$ which we have not ventured into. The choice of our space of convergence was governed by jumps in heavy-tailed processes. Namely, we expect all of the random variables to attain values on the same scale. The exact structure of such deviations remains largely an open question. Similarly, we did not explore into situations where the ``iid" assumption gets relaxed. A $j\thenum$ level LDP happens at a rate 
which is the $j\thenum$ power of the rate of the first LDP. This clearly is a consequence of the independence among the random variables driving the random walk. Future work on weakly dependent variables are under current investigation.

\section*{Acknowledgements}
We would like to thank Parthanil Roy for interesting discussions on the preliminary ideas of (hidden) large deviations.
Additionally, we gratefully acknowledge support from MOE Tier 2 grant MOE-2013-T2-1-158. 
\bibliographystyle{imsart-nameyear}
\bibliography{bib_LDQ.bib}
\end{document}